\documentclass
[numbers=enddot,12pt,final,onecolumn,notitlepage,abstracton]{scrartcl}%
\usepackage[headsepline,footsepline,manualmark]{scrlayer-scrpage}
\usepackage[all,cmtip]{xy}
\usepackage{amsfonts}
\usepackage{amssymb}
\usepackage{framed}
\usepackage{amsmath}
\usepackage{comment}
\usepackage{needspace}
\usepackage{color}
\usepackage[breaklinks=true]{hyperref}
\usepackage[sc]{mathpazo}
\usepackage[T1]{fontenc}
\usepackage{amsthm}
\usepackage{ytableau}
\usepackage{tabu}
\providecommand{\U}[1]{\protect\rule{.1in}{.1in}}
\theoremstyle{definition}
\newtheorem{theo}{Theorem}[section]
\newenvironment{theorem}[1][]
{\begin{theo}[#1]\begin{leftbar}}
{\end{leftbar}\end{theo}}
\newtheorem{lem}[theo]{Lemma}
\newenvironment{lemma}[1][]
{\begin{lem}[#1]\begin{leftbar}}
{\end{leftbar}\end{lem}}
\newtheorem{prop}[theo]{Proposition}
\newenvironment{proposition}[1][]
{\begin{prop}[#1]\begin{leftbar}}
{\end{leftbar}\end{prop}}
\newtheorem{defi}[theo]{Definition}
\newenvironment{definition}[1][]
{\begin{defi}[#1]\begin{leftbar}}
{\end{leftbar}\end{defi}}
\newtheorem{remk}[theo]{Remark}
\newenvironment{remark}[1][]
{\begin{remk}[#1]\begin{leftbar}}
{\end{leftbar}\end{remk}}
\newtheorem{coro}[theo]{Corollary}

\newtheorem{warn}[theo]{Warning}

\newtheorem{conj}[theo]{Conjecture}
\newenvironment{conjecture}[1][]
{\begin{conj}[#1]\begin{leftbar}}
{\end{leftbar}\end{conj}}
\newtheorem{exmp}[theo]{Example}

\excludecomment{verlong}
\includecomment{vershort}
\excludecomment{noncompile}
\excludecomment{obsolete}

\newcommand\arcstr{\ar@/^1pc/}
\let\sumnonlimits\sum
\let\prodnonlimits\prod
\renewcommand{\sum}{\sumnonlimits\limits}
\renewcommand{\prod}{\prodnonlimits\limits}
\ihead{Proof of BCL conjecture}
\ohead{June 13, 2026}
\begin{document}

\title{Proof of a conjecture of Bergeron, Ceballos and Labb\'e}
\author{Alexander Postnikov and Darij Grinberg}
\date{October 25, 2017 (with corrections June 13, 2026)}
\maketitle

\begin{abstract}
The reduced expressions for a given element $w$ of a Coxeter group $\left(
W,S\right)  $ can be regarded as the vertices of a directed graph
$\mathcal{R}\left(  w\right)  $; its arcs correspond to the braid moves.
Specifically, an arc goes from a reduced expression $\overrightarrow{a}$ to a
reduced expression $\overrightarrow{b}$ when $\overrightarrow{b}$ is obtained
from $\overrightarrow{a}$ by replacing a contiguous subword of the form
$stst\cdots$ (for some distinct $s,t\in S$) by $tsts\cdots$ (where both
subwords have length $m_{s,t}$, the order of $st\in W$). We prove a strong
bipartiteness-type result for this graph $\mathcal{R}\left(  w\right)  $: Not
only does every cycle of $\mathcal{R}\left(  w\right)  $ have even length;
actually, the arcs of $\mathcal{R}\left(  w\right)  $ can be colored (with
colors corresponding to the type of braid moves used), and to every color $c$
corresponds an \textquotedblleft opposite\textquotedblright\ color
$c^{\operatorname*{op}}$ (corresponding to the reverses of the braid moves
with color $c$), and for any color $c$, the number of arcs in any given cycle
of $\mathcal{R}\left(  w\right)  $ having color in $\left\{
c,c^{\operatorname*{op}}\right\}  $ is even. This is a generalization and
strengthening of a 2014 result by Bergeron, Ceballos and Labb\'{e}.

\end{abstract}

\section*{Introduction}

Let $\left(  W,S\right)  $ be a Coxeter group\footnote{All terminology and
notation that appears in this introduction will later be defined in more
detail.} with Coxeter matrix $\left(  m_{s,s^{\prime}}\right)  _{\left(
s,s^{\prime}\right)  \in S\times S}$, and let $w\in W$. Consider a directed
graph $\mathcal{R}\left(  w\right)  $ whose vertices are the reduced
expressions for $w$, and whose arcs are defined as follows: The graph
$\mathcal{R}\left(  w\right)  $ has an arc from a reduced expression
$\overrightarrow{a}$ to a reduced expression $\overrightarrow{b}$ whenever
$\overrightarrow{b}$ can be obtained from $\overrightarrow{a}$ by replacing
some contiguous subword of the form $\underbrace{\left(  s,t,s,t,\ldots
\right)  }_{m_{s,t}\text{ letters}}$ by $\underbrace{\left(  t,s,t,s,\ldots
\right)  }_{m_{s,t}\text{ letters}}$, where $s$ and $t$ are two distinct
elements of $S$. (This replacement is called an $\left(  s,t\right)
$\emph{-braid move}.)

The directed graph $\mathcal{R}\left(  w\right)  $ (or, rather, its undirected
version) has been studied many times; see, for example, \cite{ReiRoi11} and
the references therein. In this note, we shall prove a bipartiteness-type
result for $\mathcal{R}\left(  w\right)  $. Its simplest aspect (actually, a
corollary) is the fact that $\mathcal{R}\left(  w\right)  $ is bipartite
(i.e., every cycle of $\mathcal{R}\left(  w\right)  $ has even length); but we
shall concern ourselves with stronger statements. We can regard $\mathcal{R}%
\left(  w\right)  $ as an edge-colored directed graph: Namely, whenever a
reduced expression $\overrightarrow{b}$ is obtained from a reduced expression
$\overrightarrow{a}$ by an $\left(  s,t\right)  $-braid move, we color the arc
from $\overrightarrow{a}$ to $\overrightarrow{b}$ with the conjugacy
class\footnote{A \emph{conjugacy class}\ here means an equivalence class under
the relation $\sim$ on the set $S\times S$, which is given by
\[
\left(  \left(  s,t\right)  \sim\left(  s^{\prime},t^{\prime}\right)
\ \Longleftrightarrow\ \text{there exists a }q\in W\text{ such that }%
qsq^{-1}=s^{\prime}\text{ and }qtq^{-1}=t^{\prime}\right)  .
\]
The conjugacy class of an $\left(  s,t\right)  \in S\times S$ is denoted by
$\left[  \left(  s,t\right)  \right]  $.} $\left[  \left(  s,t\right)
\right]  $ of the pair $\left(  s,t\right)  \in S\times S$. Our result
(Theorem \ref{thm.BCL}) then states that, for every such color $\left[
\left(  s,t\right)  \right]  $, every cycle of $\mathcal{R}\left(  w\right)  $
has as many arcs colored $\left[  \left(  s,t\right)  \right]  $ as it has
arcs colored $\left[  \left(  t,s\right)  \right]  $, and that the total
number of arcs colored $\left[  \left(  s,t\right)  \right]  $ and $\left[
\left(  t,s\right)  \right]  $ in any given cycle is even. This generalizes
and strengthens a result of Bergeron, Ceballos and Labb\'{e} \cite[Theorem
3.1]{BCL}.

\subsection*{Acknowledgments}

We thank Nantel Bergeron and Cesar Ceballos for introducing us to the problem
at hand, and the referee for useful remarks.

\section{\label{sect.motivate-ex}A motivating example}

Before we introduce the general setting, let us demonstrate it on a simple
example. This example is not necessary for the rest of this note (and can be
skipped by the reader\footnote{All notations introduced in Section
\ref{sect.motivate-ex} should be understood as local to this section; they
will not be used beyond it (and often will be replaced by eponymic notations
for more general objects).}); it merely provides some intuition and motivation
for the definitions to come.

For this example, we fix an integer $n\geq1$, and we let $W$ be the symmetric
group $S_{n}$ of the set $\left\{  1,2,\ldots,n\right\}  $. For each
$i\in\left\{  1,2,\ldots,n-1\right\}  $, let $s_{i}\in W$ be the transposition
which switches $i$ with $i+1$ (while leaving the remaining elements of
$\left\{  1,2,\ldots,n\right\}  $ unchanged). Let $S=\left\{  s_{1}%
,s_{2},\ldots,s_{n-1}\right\}  \subseteq W$. The pair $\left(  W,S\right)  $
is an example of what is called a \emph{Coxeter group} (see, e.g.,
\cite[Chapter 4]{Bourbaki4-6} and \cite[\S 1]{Lusztig-Hecke}); more precisely,
it is known as the Coxeter group $A_{n-1}$. In particular, $S$ is a generating
set for $W$, and the group $W$ can be described by the generators $s_{1}%
,s_{2},\ldots,s_{n-1}$ and the relations%
\begin{align}
s_{i}^{2}  &  =\operatorname*{id}\ \ \ \ \ \ \ \ \ \ \text{for every }%
i\in\left\{  1,2,\ldots,n-1\right\}
;\ \ \ \ \ \ \ \ \ \ \label{eq.exam.A3.quad}\\
s_{i}s_{j}  &  =s_{j}s_{i}\ \ \ \ \ \ \ \ \ \ \text{for every }i,j\in\left\{
1,2,\ldots,n-1\right\}  \text{ such that }\left\vert i-j\right\vert
>1;\ \ \ \ \ \ \ \label{eq.exam.A3.braid1}\\
s_{i}s_{j}s_{i}  &  =s_{j}s_{i}s_{j}\ \ \ \ \ \ \ \ \ \ \text{for every
}i,j\in\left\{  1,2,\ldots,n-1\right\}  \text{ such that }\left\vert
i-j\right\vert =1.\ \ \ \ \ \label{eq.exam.A3.braid2}%
\end{align}
This is known as the \emph{Coxeter presentation} of $S_{n}$, and is due to
Moore (see, e.g., \cite[(6.23)--(6.25)]{CoxMos80} or \cite[Theorem
1.2.4]{Williamson}).

Given any $w\in W$, there exists a tuple $\left(  a_{1},a_{2},\ldots
,a_{k}\right)  $ of elements of $S$ such that $w=a_{1}a_{2}\cdots a_{k}$
(since $S$ generates $W$). Such a tuple is called a \emph{reduced expression}
for $w$ if its length $k$ is minimal among all such tuples (for the given
$w$). For instance, when $n=4$, the permutation $\pi\in S_{4}=W$ that is
written as $\left(  3,1,4,2\right)  $ in one-line notation has reduced
expressions $\left(  s_{2},s_{1},s_{3}\right)  $ and $\left(  s_{2}%
,s_{3},s_{1}\right)  $; in fact, $\pi=s_{2}s_{1}s_{3}=s_{2}s_{3}s_{1}$. (We
are following the convention by which the product $u\circ v=uv$ of two
permutations $u,v\in S_{n}$ is defined to be the permutation sending each $i$
to $u\left(  v\left(  i\right)  \right)  $.)

Given a $w\in W$, the set of reduced expressions for $w$ has an additional
structure of a directed graph. Namely, the equalities (\ref{eq.exam.A3.braid1}%
) and (\ref{eq.exam.A3.braid2}) show that, given a reduced expression
$\overrightarrow{a}=\left(  a_{1},a_{2},\ldots,a_{k}\right)  $ for $w\in W$,
we can obtain another reduced expression in any of the following two ways:

\begin{itemize}
\item Pick some $i,j\in\left\{  1,2,\ldots,n-1\right\}  $ such that
$\left\vert i-j\right\vert >1$, and pick any factor of the form $\left(
s_{i},s_{j}\right)  $ in $\overrightarrow{a}$ (that is, a pair of adjacent
entries of $\overrightarrow{a}$, the first of which is $s_{i}$ and the second
of which is $s_{j}$), provided that such a factor exists, and replace this
factor by $\left(  s_{j},s_{i}\right)  $.

\item Alternatively, pick some $i,j\in\left\{  1,2,\ldots,n-1\right\}  $ such
that $\left\vert i-j\right\vert =1$, and pick any factor of the form $\left(
s_{i},s_{j},s_{i}\right)  $ in $\overrightarrow{a}$, provided that such a
factor exists, and replace this factor by $\left(  s_{j},s_{i},s_{j}\right)  $.
\end{itemize}

In both cases, we obtain a new reduced expression for $w$ (provided that the
respective factors exist). We say that this new expression is obtained from
$\overrightarrow{a}$ by an $\left(  s_{i},s_{j}\right)  $\emph{-braid move},
or (when we do not want to mention $s_{i}$ and $s_{j}$) by a \emph{braid
move}. For instance, the reduced expression $\left(  s_{2},s_{1},s_{3}\right)
$ for $\pi=\left(  3,1,4,2\right)  \in S_{4}$ is obtained from the reduced
expression $\left(  s_{2},s_{3},s_{1}\right)  $ by an $\left(  s_{3}%
,s_{1}\right)  $-braid move, and conversely $\left(  s_{2},s_{3},s_{1}\right)
$ is obtained from $\left(  s_{2},s_{1},s_{3}\right)  $ by an $\left(
s_{1},s_{3}\right)  $-braid move.

Now, we can define a directed graph $\mathcal{R}_{0}\left(  w\right)  $ whose
vertices are the reduced expressions for $w$, and which has an edge from
$\overrightarrow{a}$ to $\overrightarrow{b}$ whenever $\overrightarrow{b}$ is
obtained from $\overrightarrow{a}$ by a braid move (of either sort). For
instance, let $n=5$, and let $w$ be the permutation written in one-line
notation as $\left(  3,2,1,5,4\right)  $. Then, $\mathcal{R}_{0}\left(
w\right)  $ looks as follows:%
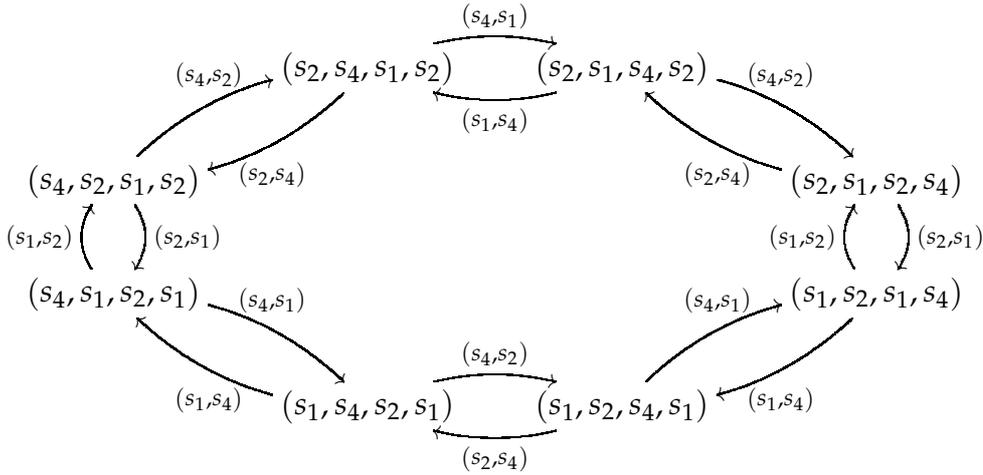
\[%
\xymatrix{
& \left(s_2,s_4,s_1,s_2\right) \arcstr[r]^{\left(s_4,s_1\right)}
\arcstr[dl]^{\left(s_2,s_4\right)} & \left(s_2,s_1,s_4,s_2\right
) \arcstr[l]^{\left(s_1,s_4\right)} \arcstr[rd]^{\left(s_4,s_2\right)} \\
\left(s_4,s_2,s_1,s_2\right) \arcstr[ur]^{\left(s_4,s_2\right)} \arcstr
[d]^{\left(s_2,s_1\right)} & & & \left(s_2,s_1,s_2,s_4\right) \arcstr
[lu]^{\left(s_2,s_4\right)}
\arcstr[d]^{\left(s_2,s_1\right)} \\
\left(s_4,s_1,s_2,s_1\right) \arcstr[u]^{\left(s_1,s_2\right)} \arcstr
[dr]^{\left(s_4,s_1\right)} & & & \left(s_1,s_2,s_1,s_4\right) \arcstr
[u]^{\left(s_1,s_2\right)}
\arcstr[dl]^{\left(s_1,s_4\right)} \\
& \left(s_1,s_4,s_2,s_1\right) \arcstr[ul]^{\left(s_1,s_4\right)}
\arcstr[r]^{\left(s_4,s_2\right)} & \left(s_1,s_2,s_4,s_1\right) \arcstr
[l]^{\left(s_2,s_4\right)} \arcstr[ur]^{\left(s_4,s_1\right)}
}%
.
\]
Here, we have \textquotedblleft colored\textquotedblright\ (i.e., labelled)
every arc $\left(  \overrightarrow{a},\overrightarrow{b}\right)  $ with the
pair $\left(  s_{i},s_{j}\right)  $ such that $\overrightarrow{b}$ is obtained
from $\overrightarrow{a}$ by an $\left(  s_{i},s_{j}\right)  $-braid move.

In our particular case, the graph $\mathcal{R}_{0}\left(  w\right)  $ consists
of a single bidirected cycle. This is not true in general, but certain things
hold in general. First, it is clear that whenever an arc from some vertex
$\overrightarrow{a}$ to some vertex $\overrightarrow{b}$ has color $\left(
s_{i},s_{j}\right)  $, then there is an arc with color $\left(  s_{j}%
,s_{i}\right)  $ from $\overrightarrow{b}$ to $\overrightarrow{a}$. Thus,
$\mathcal{R}_{0}\left(  w\right)  $ can be regarded as an undirected graph (at
the expense of murkying up the colors of the arcs). Furthermore, every reduced
expression for $w$ can be obtained from any other by a sequence of braid moves
(this is the Matsumoto-Tits theorem; it appears, e.g., in \cite[Theorem
1.9]{Lusztig-Hecke}). Thus, the graph $\mathcal{R}_{0}\left(  w\right)  $ is
strongly connected.

What do the cycles of $\mathcal{R}_{0}\left(  w\right)  $ have in common?
Walking down the long cycle in the graph $\mathcal{R}_{0}\left(  w\right)  $
for $w=\left(  3,2,1,5,4\right)  \in S_{5}$ clockwise, we observe that the
$\left(  s_{1},s_{2}\right)  $-braid move is used once (i.e., we traverse
precisely one arc with color $\left(  s_{1},s_{2}\right)  $), the $\left(
s_{2},s_{1}\right)  $-braid move once, the $\left(  s_{1},s_{4}\right)
$-braid move twice, the $\left(  s_{4},s_{1}\right)  $-braid move once, the
$\left(  s_{2},s_{4}\right)  $-braid move once, and the $\left(  s_{4}%
,s_{2}\right)  $-braid move twice. In particular:

\begin{itemize}
\item The total number of $\left(  s_{i},s_{j}\right)  $-braid moves with
$\left\vert i-j\right\vert =1$ used is even (namely, $2$).

\item The total number of $\left(  s_{i},s_{j}\right)  $-braid moves with
$\left\vert i-j\right\vert >1$ used is even (namely, $6$).
\end{itemize}

This example alone is scant evidence of any general result, but both evenness
patterns persist for general $n$, for any $w\in S_{n}$ and any directed cycle
in $\mathcal{R}_{0}\left(  w\right)  $. We can simplify the statement if we
change our coloring to a coarser one. Namely, let $\mathfrak{M}$ denote the
subset $\left\{  \left(  s,t\right)  \in S\times S\ \mid\ s\neq t\right\}
=\left\{  \left(  s_{i},s_{j}\right)  \ \mid\ i\neq j\right\}  $ of $S\times
S$. We define a binary relation $\sim$ on $\mathfrak{M}$ by%
\[
\left(  \left(  s,t\right)  \sim\left(  s^{\prime},t^{\prime}\right)
\ \Longleftrightarrow\ \text{there exists a }q\in W\text{ such that }%
qsq^{-1}=s^{\prime}\text{ and }qtq^{-1}=t^{\prime}\right)  .
\]
This relation $\sim$ is an equivalence relation; it thus gives rise to a
quotient set $\mathfrak{M}/\sim$. It is easy to see that the quotient set
$\mathfrak{M}/\sim$ has exactly two elements (for $n\geq4$): the equivalence
class of all $\left(  s_{i},s_{j}\right)  $ with $\left\vert i-j\right\vert
=1$, and the equivalence class of all $\left(  s_{i},s_{j}\right)  $ with
$\left\vert i-j\right\vert >1$. Let us now define an edge-colored directed
graph $\mathcal{R}\left(  w\right)  $ by starting with $\mathcal{R}_{0}\left(
w\right)  $, and replacing each color $\left(  s_{i},s_{j}\right)  $ by its
equivalence class $\left[  \left(  s_{i},s_{j}\right)  \right]  $. Thus, in
$\mathcal{R}\left(  w\right)  $, the arcs are colored with the (at most two)
elements of $\mathfrak{M}/\sim$. Now, our evenness patterns can be restated as
follows: For any $n\in\mathbb{N}$, any $w\in S_{n}$ and any color
$c\in\mathfrak{M}/\sim$, any directed cycle of $\mathcal{R}\left(  w\right)  $
has an even number of arcs with color $c$.

This can be generalized further to every Coxeter group, with a minor caveat.
Namely, let $\left(  W,S\right)  $ be a Coxeter group with Coxeter matrix
$\left(  m_{s,s^{\prime}}\right)  _{\left(  s,s^{\prime}\right)  \in S\times
S}$. Notions such as reduced expressions and braid moves still make sense (see
below for references and definitions). We redefine $\mathfrak{M}$ as $\left\{
\left(  s,t\right)  \in S\times S\ \mid\ s\neq t\text{ and }m_{s,t}%
<\infty\right\}  $ (since pairs $\left(  s,t\right)  $ with $m_{s,t}=\infty$
do not give rise to braid moves). Unlike in the case of $W=S_{n}$, it is not
necessarily true that $\left(  s,t\right)  \sim\left(  t,s\right)  $ for every
$\left(  s,t\right)  \in\mathfrak{M}$. We define $\left[  \left(  s,t\right)
\right]  ^{\operatorname*{op}}=\left[  \left(  t,s\right)  \right]  $. The
evenness pattern now has to be weakened as follows: For every $w\in W$ and any
color $c\in\mathfrak{M}/\sim$, any directed cycle of $\mathcal{R}\left(
w\right)  $ has an even number of arcs whose color belongs to $\left\{
c,c^{\operatorname*{op}}\right\}  $. (For $W=S_{n}$, we have
$c=c^{\operatorname*{op}}$, and thus this recovers our old evenness patterns.)
This is part of the main theorem we will prove in this note -- namely, Theorem
\ref{thm.BCL} \textbf{(b)}; it extends a result \cite[Theorem 3.1]{BCL}
obtained by Bergeron, Ceballos and Labb\'{e} by geometric means. The other
part of the main theorem (Theorem \ref{thm.BCL} \textbf{(a)}) states that any
directed cycle of $\mathcal{R}\left(  w\right)  $ has as many arcs with color
$c$ as it has arcs with color $c^{\operatorname*{op}}$.

\section{The theorem}

In the following, we shall use the notations of \cite[\S 1]{Lusztig-Hecke}
concerning Coxeter groups. (These notations are compatible with those of
\cite[Chapter 4]{Bourbaki4-6}, except that Bourbaki writes $m\left(
s,s^{\prime}\right)  $ instead of $m_{s,s^{\prime}}$, and speaks of
\textquotedblleft Coxeter systems\textquotedblright\ instead of
\textquotedblleft Coxeter groups\textquotedblright.)

Let us recall a brief definition of Coxeter groups and Coxeter matrices:

A \emph{Coxeter group} is a pair $\left(  W,S\right)  $, where $W$ is a group,
and where $S$ is a finite subset of $W$ having the following property: There
exists a matrix $\left(  m_{s,s^{\prime}}\right)  _{\left(  s,s^{\prime
}\right)  \in S\times S}\in\left\{  1,2,3,\ldots,\infty\right\}  ^{S\times S}$
such that

\begin{itemize}
\item every $s\in S$ satisfies $m_{s,s}=1$;

\item every two distinct elements $s$ and $t$ of $S$ satisfy $m_{s,t}%
=m_{t,s}\geq2$;

\item the group $W$ can be presented by the generators $S$ and the relations
\[
\left(  st\right)  ^{m_{s,t}}=1\ \ \ \ \ \ \ \ \ \ \text{for all }\left(
s,t\right)  \in S\times S\text{ satisfying }m_{s,t}\neq\infty.
\]

\end{itemize}

In this case, the matrix $\left(  m_{s,s^{\prime}}\right)  _{\left(
s,s^{\prime}\right)  \in S\times S}$ is called the \emph{Coxeter matrix} of
$\left(  W,S\right)  $. It is well-known (see, e.g., \cite[\S 1]%
{Lusztig-Hecke}\footnote{See also \cite[Chapter V, n$^{\circ}$ 4.3,
Corollaire]{Bourbaki4-6} for a proof of the existence of a Coxeter group
corresponding to a given Coxeter matrix. Note that Bourbaki's definition of a
\textquotedblleft Coxeter system\textquotedblright\ differs from our
definition of a \textquotedblleft Coxeter group\textquotedblright\ in the
extra requirement that $m_{s,t}$ be the order of $st\in W$; but this turns out
to be a consequence of the other requirements.}) that any Coxeter group has a
unique Coxeter matrix, and conversely, for every finite set $S$ and any matrix
$\left(  m_{s,s^{\prime}}\right)  _{\left(  s,s^{\prime}\right)  \in S\times
S}\in\left\{  1,2,3,\ldots,\infty\right\}  ^{S\times S}$ satisfying the first
two of the three requirements above, there exists a unique (up to isomorphism
preserving $S$) Coxeter group $\left(  W,S\right)  $.

We fix a Coxeter group $\left(  W,S\right)  $ with Coxeter matrix $\left(
m_{s,s^{\prime}}\right)  _{\left(  s,s^{\prime}\right)  \in S\times S}$. Thus,
$W$ is a group, and $S$ is a set of elements of order $2$ in $W$ such that for
every $\left(  s,s^{\prime}\right)  \in S\times S$, the element $ss^{\prime
}\in W$ has order $m_{s,s^{\prime}}$. (See, e.g., \cite[Proposition
1.3(b)]{Lusztig-Hecke} for this well-known fact.)

We let $\mathfrak{M}$ denote the subset%
\[
\left\{  \left(  s,t\right)  \in S\times S\ \mid\ s\neq t\text{ and }%
m_{s,t}<\infty\right\}
\]
of $S\times S$. (This is denoted by $I$ in \cite[Chapter 4, n$^{\circ}$
1.3]{Bourbaki4-6}.) We define a binary relation $\sim$ on $\mathfrak{M}$ by%
\[
\left(  \left(  s,t\right)  \sim\left(  s^{\prime},t^{\prime}\right)
\ \Longleftrightarrow\ \text{there exists a }q\in W\text{ such that }%
qsq^{-1}=s^{\prime}\text{ and }qtq^{-1}=t^{\prime}\right)  .
\]
It is clear that this relation $\sim$ is an equivalence relation; it thus
gives rise to a quotient set $\mathfrak{M}/\sim$. For every pair
$P\in\mathfrak{M}$, we denote by $\left[  P\right]  $ the equivalence class of
$P$ with respect to this relation $\sim$.

We set $\mathbb{N}=\left\{  0,1,2,\ldots\right\}  $.

A \emph{word} will mean a $k$-tuple for some $k\in\mathbb{N}$. A
\emph{subword} of a word $\left(  s_{1},s_{2},\ldots,s_{k}\right)  $ will mean
a word of the form $\left(  s_{i_{1}},s_{i_{2}},\ldots,s_{i_{p}}\right)  $,
where $i_{1},i_{2},\ldots,i_{p}$ are elements of $\left\{  1,2,\ldots
,k\right\}  $ satisfying $i_{1}<i_{2}<\cdots<i_{p}$. For instance, $\left(
1\right)  $, $\left(  3,5\right)  $, $\left(  1,3,5\right)  $, $\left(
{}\right)  $ and $\left(  1,5\right)  $ are subwords of the word $\left(
1,3,5\right)  $. A \emph{factor} of a word $\left(  s_{1},s_{2},\ldots
,s_{k}\right)  $ will mean a word of the form $\left(  s_{i+1},s_{i+2}%
,\ldots,s_{i+m}\right)  $ for some $i\in\left\{  0,1,\ldots,k\right\}  $ and
some $m\in\left\{  0,1,\ldots,k-i\right\}  $. For instance, $\left(  1\right)
$, $\left(  3,5\right)  $, $\left(  1,3,5\right)  $ and $\left(  {}\right)  $
are factors of the word $\left(  1,3,5\right)  $, but $\left(  1,5\right)  $
is not.

We recall that a \emph{reduced expression} for an element $w\in W$ is a
$k$-tuple $\left(  s_{1},s_{2},\ldots,s_{k}\right)  $ of elements of $S$ such
that $w=s_{1}s_{2}\cdots s_{k}$, and such that $k$ is minimum (among all such
tuples). The length of a reduced expression for $w$ is called the
\emph{length} of $w$, and is denoted by $l\left(  w\right)  $. Thus, a reduced
expression for an element $w\in W$ is a $k$-tuple $\left(  s_{1},s_{2}%
,\ldots,s_{k}\right)  $ of elements of $S$ such that $w=s_{1}s_{2}\cdots
s_{k}$ and $k=l\left(  w\right)  $.

\begin{definition}
\label{def.braid}Let $w\in W$. Let $\overrightarrow{a}=\left(  a_{1}%
,a_{2},\ldots,a_{k}\right)  $ and $\overrightarrow{b}=\left(  b_{1}%
,b_{2},\ldots,b_{k}\right)  $ be two reduced expressions for $w$.

Let $\left(  s,t\right)  \in\mathfrak{M}$. We say that $\overrightarrow{b}$ is
obtained from $\overrightarrow{a}$ by an $\left(  s,t\right)  $\emph{-braid
move} if $\overrightarrow{b}$ can be obtained from $\overrightarrow{a}$ by
finding a factor of $\overrightarrow{a}$ of the form $\underbrace{\left(
s,t,s,t,s,\ldots\right)  }_{m_{s,t}\text{ elements}}$ and replacing it by
$\underbrace{\left(  t,s,t,s,t,\ldots\right)  }_{m_{s,t}\text{ elements}}$.

We notice that if $\overrightarrow{b}$ is obtained from $\overrightarrow{a}$
by an $\left(  s,t\right)  $-braid move, then $\overrightarrow{a}$ is obtained
from $\overrightarrow{b}$ by an $\left(  t,s\right)  $-braid move.
\end{definition}

\begin{definition}
\label{def.R}Let $w\in W$. We define an edge-colored directed graph
$\mathcal{R}\left(  w\right)  $, whose arcs are colored with elements of
$\mathfrak{M}/\sim$, as follows:

\begin{itemize}
\item The vertex set of $\mathcal{R}\left(  w\right)  $ shall be the set of
all reduced expressions for $w$.

\item The arcs of $\mathcal{R}\left(  w\right)  $ are defined as follows:
Whenever $\left(  s,t\right)  \in\mathfrak{M}$, and whenever
$\overrightarrow{a}$ and $\overrightarrow{b}$ are two reduced expressions for
$w$ such that $\overrightarrow{b}$ is obtained from $\overrightarrow{a}$ by an
$\left(  s,t\right)  $-braid move, we draw an arc from $\overrightarrow{a}$ to
$\overrightarrow{b}$ with color $\left[  \left(  s,t\right)  \right]  $.
\end{itemize}
\end{definition}

\begin{theorem}
\label{thm.BCL}Let $w\in W$. Let $C$ be a (directed) cycle in the graph
$\mathcal{R}\left(  w\right)  $. Let $c=\left[  \left(  s,t\right)  \right]
\in\mathfrak{M}/\sim$ be an equivalence class with respect to $\sim$. Let
$c^{\operatorname*{op}}$ be the equivalence class $\left[  \left(  t,s\right)
\right]  \in\mathfrak{M}/\sim$. Then:

\begin{enumerate}
\item[\textbf{(a)}] The number of arcs colored $c$ appearing in the cycle $C$
equals the number of arcs colored $c^{\operatorname*{op}}$ appearing in the
cycle $C$.

\item[\textbf{(b)}] The number of arcs whose color belongs to $\left\{
c,c^{\operatorname*{op}}\right\}  $ appearing in the cycle $C$ is even.
\end{enumerate}
\end{theorem}

None of the parts \textbf{(a)} and \textbf{(b)} of Theorem \ref{thm.BCL} is a
trivial consequence of the other: When $c=c^{\operatorname*{op}}$, the
statement of Theorem \ref{thm.BCL} \textbf{(a)} is obvious and does not imply
part \textbf{(b)}.

Theorem \ref{thm.BCL} \textbf{(b)} generalizes \cite[Theorem 3.1]{BCL} in two
directions: First, Theorem \ref{thm.BCL} is stated for arbitrary Coxeter
groups, rather than only for finite Coxeter groups as in \cite{BCL}. Second,
in the terms of \cite[Remark 3.3]{BCL}, we are working with sets $Z$ that are
\textquotedblleft stabled by conjugation instead of
automorphism\textquotedblright.

\section{Inversions and the word $\rho_{s,t}$}

We shall now introduce some notations and state some auxiliary results that
will be used to prove Theorem \ref{thm.BCL}. Our strategy of proof is inspired
by that used in \cite[\S 3.4]{BCL} and thus (indirectly) also by that in
\cite[\S 3, and proof of Corollary 5.2]{ReiRoi11}; however, we shall avoid any
use of geometry (such as roots and hyperplane arrangements), and work entirely
with the Coxeter group itself.

We denote the subset $\bigcup\limits_{x\in W}xSx^{-1}$ of $W$ by $T$. The
elements of $T$ are called the \emph{reflections} (of $W$). They all have
order $2$. (The notation $T$ is used here in the same meaning as in
\cite[\S 1]{Lusztig-Hecke} and in \cite[Chapter 4, n$^{\circ}$ 1.4]%
{Bourbaki4-6}.)

\begin{definition}
\label{def.biset}For every $k\in\mathbb{N}$, we consider the set $W^{k}$ as a
left $W$-set by the rule%
\[
w\left(  w_{1},w_{2},\ldots,w_{k}\right)  =\left(  ww_{1},ww_{2},\ldots
,ww_{k}\right)  ,
\]
and as a right $W$-set by the rule%
\[
\left(  w_{1},w_{2},\ldots,w_{k}\right)  w=\left(  w_{1}w,w_{2}w,\ldots
,w_{k}w\right)  .
\]

\end{definition}

\begin{definition}
Let $s$ and $t$ be two distinct elements of $T$. Let $m_{s,t}$ denote the
order of the element $st\in W$. (This extends the definition of $m_{s,t}$ for
$s,t\in S$. Note that the distinctness of $s$ and $t$ entails $st\neq1$ (since
$s$ and $t$ have order $2$) and thus $m_{s,t}\geq2$.) Assume that
$m_{s,t}<\infty$. We let $D_{s,t}$ denote the subgroup of $W$ generated by $s$
and $t$. Then, $D_{s,t}$ is a dihedral group (since $s$ and $t$ are two
distinct nontrivial involutions, and since any group generated by two distinct
nontrivial involutions is dihedral). We denote by $\rho_{s,t}$ the word%
\[
\left(  \left(  st\right)  ^{0}s,\left(  st\right)  ^{1}s,\ldots,\left(
st\right)  ^{m_{s,t}-1}s\right)  =\left(  s,sts,ststs,\ldots
,\underbrace{ststs\cdots s}_{2m_{s,t}-1\text{ letters}}\right)  \in\left(
D_{s,t}\right)  ^{m_{s,t}}.
\]
Note that this word $\rho_{s,t}$ uniquely determines $s$ and $t$ (since it has
$m_{s,t}\geq2$ letters, and its first two letters are $s$ and $sts$, from
which we can easily reconstruct $s$ and $t$). In other words, the words
$\rho_{s,t}$ for different pairs $\left(  s,t\right)  $ are distinct.
\end{definition}

The \emph{reversal} of a word $\left(  a_{1},a_{2},\ldots,a_{k}\right)  $ is
defined to be the word $\left(  a_{k},a_{k-1},\ldots,a_{1}\right)  $.

The following proposition collects some simple properties of the words
$\rho_{s,t}$.

\begin{proposition}
\label{prop.rhost}Let $s$ and $t$ be two distinct elements of $T$ such that
$m_{s,t}<\infty$. Then:

\begin{enumerate}
\item[\textbf{(a)}] The word $\rho_{s,t}$ consists of reflections in $D_{s,t}%
$, and contains every reflection in $D_{s,t}$ exactly once.

\item[\textbf{(b)}] The word $\rho_{t,s}$ is the reversal of the word
$\rho_{s,t}$.

\item[\textbf{(c)}] Let $q\in W$. Then, the word $q\rho_{t,s}q^{-1}$ is the
reversal of the word $q\rho_{s,t}q^{-1}$.
\end{enumerate}
\end{proposition}

\begin{proof}
[Proof of Proposition \ref{prop.rhost}.]\textbf{(a)} We need to prove three claims:

\textit{Claim 1:} Every entry of the word $\rho_{s,t}$ is a reflection in
$D_{s,t}$.

\textit{Claim 2:} The entries of the word $\rho_{s,t}$ are distinct.

\textit{Claim 3:} Every reflection in $D_{s,t}$ is an entry of the word
$\rho_{s,t}$.

\textit{Proof of Claim 1:} We must show that $\left(  st\right)  ^{k}s$ is a
reflection in $D_{s,t}$ for every $k\in\left\{  0,1,\ldots,m_{s,t}-1\right\}
$. Thus, fix $k\in\left\{  0,1,\ldots,m_{s,t}-1\right\}  $. Then,%
\begin{align*}
\left(  st\right)  ^{k}s  &  =\underbrace{stst\cdots s}_{2k+1\text{ letters}}=%
\begin{cases}
\underbrace{stst\cdots t}_{k\text{ letters}}s\underbrace{tsts\cdots
s}_{k\text{ letters}}, & \text{if }k\text{ is even};\\
\underbrace{stst\cdots s}_{k\text{ letters}}t\underbrace{stst\cdots
s}_{k\text{ letters}}, & \text{if }k\text{ is odd}%
\end{cases}
\\
&  =%
\begin{cases}
\underbrace{stst\cdots t}_{k\text{ letters}}s\left(  \underbrace{stst\cdots
t}_{k\text{ letters}}\right)  ^{-1}, & \text{if }k\text{ is even};\\
\underbrace{stst\cdots s}_{k\text{ letters}}t\left(  \underbrace{stst\cdots
s}_{k\text{ letters}}\right)  ^{-1}, & \text{if }k\text{ is odd}%
\end{cases}
\\
&  \ \ \ \ \ \ \ \ \ \ \left(
\begin{array}
[c]{c}%
\text{since }\underbrace{tsts\cdots s}_{k\text{ letters}}=\left(
\underbrace{stst\cdots t}_{k\text{ letters}}\right)  ^{-1}\text{ if }k\text{
is even,}\\
\text{and }\underbrace{stst\cdots s}_{k\text{ letters}}=\left(
\underbrace{stst\cdots s}_{k\text{ letters}}\right)  ^{-1}\text{ if }k\text{
is odd}%
\end{array}
\right)  .
\end{align*}
Hence, $\left(  st\right)  ^{k}s$ is conjugate to either $s$ or $t$ (depending
on whether $k$ is even or odd). Thus, $\left(  st\right)  ^{k}s$ is a
reflection. Also, it clearly lies in $D_{s,t}$. This proves Claim 1.

\textit{Proof of Claim 2:} The element $st$ of $W$ has order $m_{s,t}$. Thus,
the elements $\left(  st\right)  ^{0},\left(  st\right)  ^{1},\ldots,\left(
st\right)  ^{m_{s,t}-1}$ are all distinct. Hence, the elements $\left(
st\right)  ^{0}s,\left(  st\right)  ^{1}s,\ldots,\left(  st\right)
^{m_{s,t}-1}s$ are all distinct. In other words, the entries of the word
$\rho_{s,t}$ are all distinct. Claim 2 is proven.

\textit{Proof of Claim 3:} The dihedral group $D_{s,t}$ has $2m_{s,t}$
elements\footnote{since it is generated by two distinct involutions $s\neq1$
and $t\neq1$ whose product $st$ has order $m_{s,t}$}, of which at most
$m_{s,t}$ are reflections\footnote{\textit{Proof.} Consider the group
homomorphism $\operatorname*{sgn}:W\rightarrow\left\{  1,-1\right\}  $ defined
in \cite[\S 1.1]{Lusztig-Hecke}. The group homomorphism $\operatorname*{sgn}%
\mid_{D_{s,t}}:D_{s,t}\rightarrow\left\{  1,-1\right\}  $ sends either none or
$m_{s,t}$ elements of $D_{s,t}$ to $-1$. Thus, this homomorphism
$\operatorname*{sgn}\mid_{D_{s,t}}$ sends at most $m_{s,t}$ elements of
$D_{s,t}$ to $-1$. Since it must send every reflection to $-1$, this shows
that at most $m_{s,t}$ elements of $D_{s,t}$ are reflections.
\par
(Actually, we can replace \textquotedblleft at most\textquotedblright\ by
\textquotedblleft exactly\textquotedblright\ here, but we won't need this.)}.
But the word $\rho_{s,t}$ has $m_{s,t}$ entries, and all its entries are
reflections in $D_{s,t}$ (by Claim 1); hence, it contains $m_{s,t}$
reflections in $D_{s,t}$ (by Claim 2). Since $D_{s,t}$ has only at most
$m_{s,t}$ reflections, this shows that every reflection in $D_{s,t}$ is an
entry of the word $\rho_{s,t}$. Claim 3 is proven.

This finishes the proof of Proposition \ref{prop.rhost} \textbf{(a)}. \medskip

\textbf{(b)} We have $\rho_{s,t}=\left(  \left(  st\right)  ^{0}s,\left(
st\right)  ^{1}s,\ldots,\left(  st\right)  ^{m_{s,t}-1}s\right)  $ and
\newline$\rho_{t,s}=\left(  \left(  ts\right)  ^{0}t,\left(  ts\right)
^{1}t,\ldots,\left(  ts\right)  ^{m_{s,t}-1}t\right)  $ (since $m_{t,s}%
=m_{s,t}$). Thus, in order to prove Proposition \ref{prop.rhost} \textbf{(b)},
we must merely show that $\left(  st\right)  ^{k}s=\left(  ts\right)
^{m_{s,t}-1-k}t$ for every $k\in\left\{  0,1,\ldots,m_{s,t}-1\right\}  $.

So fix $k\in\left\{  0,1,\ldots,m_{s,t}-1\right\}  $. Then,%
\begin{align*}
\left(  st\right)  ^{k}s\cdot\left(  \left(  ts\right)  ^{m_{s,t}%
-1-k}t\right)  ^{-1}  &  =\left(  st\right)  ^{k}s\underbrace{t^{-1}}%
_{=t}\underbrace{\left(  \left(  ts\right)  ^{m_{s,t}-1-k}\right)  ^{-1}%
}_{=\left(  s^{-1}t^{-1}\right)  ^{m_{s,t}-1-k}}=\underbrace{\left(
st\right)  ^{k}st}_{=\left(  st\right)  ^{k+1}}\left(  \underbrace{s^{-1}%
}_{=s}\underbrace{t^{-1}}_{=t}\right)  ^{m_{s,t}-1-k}\\
&  =\left(  st\right)  ^{k+1}\left(  st\right)  ^{m_{s,t}-1-k}=\left(
st\right)  ^{m_{s,t}}=1,
\end{align*}
so that $\left(  st\right)  ^{k}s=\left(  ts\right)  ^{m_{s,t}-1-k}t$. This
proves Proposition \ref{prop.rhost} \textbf{(b)}. \medskip

\textbf{(c)} Let $q\in W$. Proposition \ref{prop.rhost} \textbf{(b)} shows
that the word $\rho_{t,s}$ is the reversal of the word $\rho_{s,t}$. Hence,
the word $q\rho_{t,s}q^{-1}$ is the reversal of the word $q\rho_{s,t}q^{-1}$
(since the word $q\rho_{t,s}q^{-1}$ is obtained from $\rho_{t,s}$ by
conjugating each letter by $q$, and the word $q\rho_{s,t}q^{-1}$ is obtained
from $\rho_{s,t}$ in the same way). This proves Proposition \ref{prop.rhost}
\textbf{(c)}.
\end{proof}

\begin{definition}
\label{def.Invsle}Let $\overrightarrow{a}=\left(  a_{1},a_{2},\ldots
,a_{k}\right)  \in S^{k}$. Then, $\operatorname*{Invs}\overrightarrow{a}$ is
defined to be the $k$-tuple $\left(  t_{1},t_{2},\ldots,t_{k}\right)  \in
T^{k}$, where we set%
\[
t_{i}=\left(  a_{1}a_{2}\cdots a_{i-1}\right)  a_{i}\left(  a_{1}a_{2}\cdots
a_{i-1}\right)  ^{-1}\ \ \ \ \ \ \ \ \ \ \text{for every }i\in\left\{
1,2,\ldots,k\right\}  .
\]

\end{definition}

\begin{remark}
Let $w\in W$. Let $\overrightarrow{a}=\left(  a_{1},a_{2},\ldots,a_{k}\right)
$ be a reduced expression for $w$. The $k$-tuple $\operatorname*{Invs}%
\overrightarrow{a}$ is denoted by $\Phi\left(  \overrightarrow{a}\right)  $ in
\cite[Chapter 4, n$^{\circ}$ 1.4]{Bourbaki4-6}, and is closely connected to
various standard constructions in Coxeter group theory. A well-known fact
states that the set of all entries of $\operatorname*{Invs}\overrightarrow{a}$
depends only on $w$ (but not on $\overrightarrow{a}$); this set is called the
\emph{(left) inversion set} of $w$. The $k$-tuple $\operatorname*{Invs}%
\overrightarrow{a}$ contains each element of this set exactly once (see
Proposition \ref{prop.Invsles} below); it thus induces a total order on this set.
\end{remark}

\begin{proposition}
\label{prop.Invsles}Let $w\in W$.

\begin{enumerate}
\item[\textbf{(a)}] If $\overrightarrow{a}$ is a reduced expression for $w$,
then all entries of the tuple $\operatorname*{Invs}\overrightarrow{a}$ are distinct.

\item[\textbf{(b)}] Let $\left(  s,t\right)  \in\mathfrak{M}$. Let
$\overrightarrow{a}$ and $\overrightarrow{b}$ be two reduced expressions for
$w$ such that $\overrightarrow{b}$ is obtained from $\overrightarrow{a}$ by an
$\left(  s,t\right)  $-braid move. Then, there exists a $q\in W$ such that
$\operatorname*{Invs}\overrightarrow{b}$ is obtained from
$\operatorname*{Invs}\overrightarrow{a}$ by replacing a particular factor of
the form $q\rho_{s,t}q^{-1}$ by its reversal\footnotemark.
\end{enumerate}
\end{proposition}

\footnotetext{See Definition \ref{def.biset} for the meaning of $q\rho
_{s,t}q^{-1}$.}

\begin{proof}
[Proof of Proposition \ref{prop.Invsles}.]Let $\overrightarrow{a}$ be a
reduced expression for $w$. Write $\overrightarrow{a}$ as $\left(  a_{1}%
,a_{2},\ldots,a_{k}\right)  $. Then, the definition of $\operatorname*{Invs}%
\overrightarrow{a}$ shows that $\operatorname*{Invs}\overrightarrow{a}=\left(
t_{1},t_{2},\ldots,t_{k}\right)  $, where the $t_{i}$ are defined by%
\[
t_{i}=\left(  a_{1}a_{2}\cdots a_{i-1}\right)  a_{i}\left(  a_{1}a_{2}\cdots
a_{i-1}\right)  ^{-1}\ \ \ \ \ \ \ \ \ \ \text{for every }i\in\left\{
1,2,\ldots,k\right\}  .
\]
Now, every $i\in\left\{  1,2,\ldots,k\right\}  $ satisfies%
\begin{align*}
t_{i}  &  =\left(  a_{1}a_{2}\cdots a_{i-1}\right)  a_{i}\underbrace{\left(
a_{1}a_{2}\cdots a_{i-1}\right)  ^{-1}}_{\substack{=a_{i-1}^{-1}a_{i-2}%
^{-1}\cdots a_{1}^{-1}=a_{i-1}a_{i-2}\cdots a_{1}\\\text{(since each }%
a_{j}\text{ belongs to }S\text{)}}}=\left(  a_{1}a_{2}\cdots a_{i-1}\right)
a_{i}\left(  a_{i-1}a_{i-2}\cdots a_{1}\right) \\
&  =a_{1}a_{2}\cdots a_{i-1}a_{i}a_{i-1}\cdots a_{2}a_{1}.
\end{align*}
But \cite[Proposition 1.6 (a)]{Lusztig-Hecke} (applied to $q=k$ and
$s_{i}=a_{i}$) shows that the elements $a_{1},a_{1}a_{2}a_{1},a_{1}a_{2}%
a_{3}a_{2}a_{1},\ldots,a_{1}a_{2}\cdots a_{k-1}a_{k}a_{k-1}\cdots a_{2}a_{1}$
are distinct\footnote{This also follows from \cite[Chapter 4, n$^{\circ}$ 1.4,
Lemme 2]{Bourbaki4-6}.}. In other words, the elements $t_{1},t_{2}%
,\ldots,t_{k}$ are distinct (since \newline$t_{i}=a_{1}a_{2}\cdots
a_{i-1}a_{i}a_{i-1}\cdots a_{2}a_{1}$ for every $i\in\left\{  1,2,\ldots
,k\right\}  $). In other words, all entries of the tuple $\operatorname*{Invs}%
\overrightarrow{a}$ are distinct. Proposition \ref{prop.Invsles} \textbf{(a)}
is proven. \medskip

\textbf{(b)} We need to prove that there exists a $q\in W$ such that
$\operatorname*{Invs}\overrightarrow{b}$ is obtained from
$\operatorname*{Invs}\overrightarrow{a}$ by replacing a particular factor of
the form $q\rho_{s,t}q^{-1}$ by its reversal.

We set $m=m_{s,t}$ (for the sake of brevity).

Write $\overrightarrow{a}$ as $\left(  a_{1},a_{2},\ldots,a_{k}\right)  $.

The word $\overrightarrow{b}$ can be obtained from $\overrightarrow{a}$ by an
$\left(  s,t\right)  $-braid move. In other words, the word
$\overrightarrow{b}$ can be obtained from $\overrightarrow{a}$ by finding a
factor of $\overrightarrow{a}$ of the form $\underbrace{\left(
s,t,s,t,s,\ldots\right)  }_{m\text{ elements}}$ and replacing it by
$\underbrace{\left(  t,s,t,s,t,\ldots\right)  }_{m\text{ elements}}$ (by the
definition of an \textquotedblleft$\left(  s,t\right)  $-braid
move\textquotedblright, since $m_{s,t}=m$). In other words, there exists an
$p\in\left\{  0,1,\ldots,k-m\right\}  $ such that $\left(  a_{p+1}%
,a_{p+2},\ldots,a_{p+m}\right)  =\underbrace{\left(  s,t,s,t,s,\ldots\right)
}_{m\text{ elements}}$, and the word $\overrightarrow{b}$ can be obtained by
replacing the $\left(  p+1\right)  $-st through $\left(  p+m\right)  $-th
entries of $\overrightarrow{a}$ by $\underbrace{\left(  t,s,t,s,t,\ldots
\right)  }_{m\text{ elements}}$. Consider this $p$. Write $\overrightarrow{b}$
as $\left(  b_{1},b_{2},\ldots,b_{k}\right)  $ (this is possible since the
tuple $\overrightarrow{b}$ has the same length as $\overrightarrow{a}$). Thus,%
\begin{align}
\left(  a_{1},a_{2},\ldots,a_{p}\right)   &  =\left(  b_{1},b_{2},\ldots
,b_{p}\right)  ,\label{pf.prop.Invsles.b.1}\\
\left(  a_{p+1},a_{p+2},\ldots,a_{p+m}\right)   &  =\underbrace{\left(
s,t,s,t,s,\ldots\right)  }_{m\text{ elements}},\label{pf.prop.Invsles.b.2}\\
\left(  b_{p+1},b_{p+2},\ldots,b_{p+m}\right)   &  =\underbrace{\left(
t,s,t,s,t,\ldots\right)  }_{m\text{ elements}},\label{pf.prop.Invsles.b.3}\\
\left(  a_{p+m+1},a_{p+m+2},\ldots,a_{k}\right)   &  =\left(  b_{p+m+1}%
,b_{p+m+2},\ldots,b_{k}\right)  . \label{pf.prop.Invsles.b.4}%
\end{align}
Write the $k$-tuples $\operatorname*{Invs}\overrightarrow{a}$ and
$\operatorname*{Invs}\overrightarrow{b}$ as $\left(  \alpha_{1},\alpha
_{2},\ldots,\alpha_{k}\right)  $ and $\left(  \beta_{1},\beta_{2},\ldots
,\beta_{k}\right)  $, respectively. Their definitions show that%
\begin{equation}
\alpha_{i}=\left(  a_{1}a_{2}\cdots a_{i-1}\right)  a_{i}\left(  a_{1}%
a_{2}\cdots a_{i-1}\right)  ^{-1} \label{pf.prop.Invsles.b.alpha}%
\end{equation}
and%
\begin{equation}
\beta_{i}=\left(  b_{1}b_{2}\cdots b_{i-1}\right)  b_{i}\left(  b_{1}%
b_{2}\cdots b_{i-1}\right)  ^{-1} \label{pf.prop.Invsles.b.beta}%
\end{equation}
for every $i\in\left\{  1,2,\ldots,k\right\}  $.

Now, set $q=a_{1}a_{2}\cdots a_{p}$. From (\ref{pf.prop.Invsles.b.1}), we see
that $q=b_{1}b_{2}\cdots b_{p}$ as well. In order to prove Proposition
\ref{prop.Invsles} \textbf{(b)}, it clearly suffices to show that
$\operatorname*{Invs}\overrightarrow{b}$ is obtained from
$\operatorname*{Invs}\overrightarrow{a}$ by replacing a particular factor of
the form $q\rho_{s,t}q^{-1}$ -- namely, the factor $\left(  \alpha
_{p+1},\alpha_{p+2},\ldots,\alpha_{p+m}\right)  $ -- by its reversal.

So let us show this. In view of $\operatorname*{Invs}\overrightarrow{a}%
=\left(  \alpha_{1},\alpha_{2},\ldots,\alpha_{k}\right)  $ and
$\operatorname*{Invs}\overrightarrow{b}=\left(  \beta_{1},\beta_{2}%
,\ldots,\beta_{k}\right)  $, it clearly suffices to prove the following claims:

\textit{Claim 1:} We have $\beta_{i}=\alpha_{i}$ for every $i\in\left\{
1,2,\ldots,p\right\}  $.

\textit{Claim 2:} We have $\left(  \alpha_{p+1},\alpha_{p+2},\ldots
,\alpha_{p+m}\right)  =q\rho_{s,t}q^{-1}$.

\textit{Claim 3:} The $m$-tuple $\left(  \beta_{p+1},\beta_{p+2},\ldots
,\beta_{p+m}\right)  $ is the reversal of $\left(  \alpha_{p+1},\alpha
_{p+2},\ldots,\alpha_{p+m}\right)  $.

\textit{Claim 4:} We have $\beta_{i}=\alpha_{i}$ for every $i\in\left\{
p+m+1,p+m+2,\ldots,k\right\}  $.

\textit{Proof of Claim 1:} Let $i\in\left\{  1,2,\ldots,p\right\}  $. Then,
(\ref{pf.prop.Invsles.b.1}) shows that $a_{g}=b_{g}$ for every $g\in\left\{
1,2,\ldots,i\right\}  $. Now, (\ref{pf.prop.Invsles.b.alpha}) becomes%
\begin{align*}
\alpha_{i}  &  =\left(  a_{1}a_{2}\cdots a_{i-1}\right)  a_{i}\left(
a_{1}a_{2}\cdots a_{i-1}\right)  ^{-1}=\left(  b_{1}b_{2}\cdots b_{i-1}%
\right)  b_{i}\left(  b_{1}b_{2}\cdots b_{i-1}\right)  ^{-1}\\
&  \ \ \ \ \ \ \ \ \ \ \left(  \text{since }a_{g}=b_{g}\text{ for every }%
g\in\left\{  1,2,\ldots,i\right\}  \right) \\
&  =\beta_{i}\ \ \ \ \ \ \ \ \ \ \left(  \text{by
(\ref{pf.prop.Invsles.b.beta})}\right)  .
\end{align*}
This proves Claim 1.

\textit{Proof of Claim 2:} We have%
\[
\rho_{s,t}=\left(  \left(  st\right)  ^{0}s,\left(  st\right)  ^{1}%
s,\ldots,\left(  st\right)  ^{m_{s,t}-1}s\right)  =\left(  \left(  st\right)
^{0}s,\left(  st\right)  ^{1}s,\ldots,\left(  st\right)  ^{m-1}s\right)
\]
(since $m_{s,t}=m$). Hence,%
\begin{align*}
q\rho_{s,t}q^{-1}  &  =q\left(  \left(  st\right)  ^{0}s,\left(  st\right)
^{1}s,\ldots,\left(  st\right)  ^{m-1}s\right)  q^{-1}\\
&  =\left(  q\left(  st\right)  ^{0}sq^{-1},q\left(  st\right)  ^{1}%
sq^{-1},\ldots,q\left(  st\right)  ^{m-1}sq^{-1}\right)  .
\end{align*}
Thus, in order to prove $\left(  \alpha_{p+1},\alpha_{p+2},\ldots,\alpha
_{p+m}\right)  =q\rho_{s,t}q^{-1}$, it suffices to show that $\alpha
_{p+i}=q\left(  st\right)  ^{i-1}sq^{-1}$ for every $i\in\left\{
1,2,\ldots,m\right\}  $. So let us fix $i\in\left\{  1,2,\ldots,m\right\}  $.

We have%
\[
a_{1}a_{2}\cdots a_{p+i-1}=\underbrace{\left(  a_{1}a_{2}\cdots a_{p}\right)
}_{=q}\underbrace{\left(  a_{p+1}a_{p+2}\cdots a_{p+i-1}\right)
}_{\substack{=\underbrace{stst\cdots}_{i-1\text{ letters}}\\\text{(by
(\ref{pf.prop.Invsles.b.2}))}}}=q\underbrace{stst\cdots}_{i-1\text{ letters}%
}.
\]
Hence,%
\begin{align*}
\left(  a_{1}a_{2}\cdots a_{p+i-1}\right)  ^{-1}  &  =\left(
q\underbrace{stst\cdots}_{i-1\text{ letters}}\right)  ^{-1}=\underbrace{\cdots
t^{-1}s^{-1}t^{-1}s^{-1}}_{i-1\text{ letters}}q^{-1}\\
&  =\underbrace{\cdots tsts}_{i-1\text{ letters}}q^{-1}%
\ \ \ \ \ \ \ \ \ \ \left(  \text{since }s^{-1}=s\text{ and }t^{-1}=t\right)
.
\end{align*}
Also,%
\[
\left(  a_{1}a_{2}\cdots a_{p+i-1}\right)  a_{p+i}=a_{1}a_{2}\cdots
a_{p+i}=\underbrace{\left(  a_{1}a_{2}\cdots a_{p}\right)  }_{=q}%
\underbrace{\left(  a_{p+1}a_{p+2}\cdots a_{p+i}\right)  }%
_{\substack{=\underbrace{stst\cdots}_{i\text{ letters}}\\\text{(by
(\ref{pf.prop.Invsles.b.2}))}}}=q\underbrace{stst\cdots}_{i\text{ letters}}.
\]
Now, (\ref{pf.prop.Invsles.b.alpha}) (applied to $p+i$ instead of $i$) yields%
\begin{align*}
\alpha_{p+i}  &  =\underbrace{\left(  a_{1}a_{2}\cdots a_{p+i-1}\right)
a_{p+i}}_{=q\underbrace{stst\cdots}_{i\text{ letters}}}\underbrace{\left(
a_{1}a_{2}\cdots a_{p+i-1}\right)  ^{-1}}_{=\underbrace{\cdots tsts}%
_{i-1\text{ letters}}q^{-1}}=q\underbrace{\underbrace{stst\cdots}_{i\text{
letters}}\underbrace{\cdots tsts}_{i-1\text{ letters}}}%
_{=\underbrace{stst\cdots s}_{2i-1\text{ letters}}=\left(  st\right)  ^{i-1}%
s}q^{-1}\\
&  =q\left(  st\right)  ^{i-1}sq^{-1}.
\end{align*}
This completes the proof of $\left(  \alpha_{p+1},\alpha_{p+2},\ldots
,\alpha_{p+m}\right)  =q\rho_{s,t}q^{-1}$. Hence, Claim 2 is proven.

\textit{Proof of Claim 3:} In our proof of Claim 2, we have shown that
$\left(  \alpha_{p+1},\alpha_{p+2},\ldots,\alpha_{p+m}\right)  =q\rho
_{s,t}q^{-1}$. The same argument (applied to $\overrightarrow{b}$, $\left(
b_{1},b_{2},\ldots,b_{k}\right)  $, $\left(  \beta_{1},\beta_{2},\ldots
,\beta_{k}\right)  $, $t$ and $s$ instead of $\overrightarrow{a}$, $\left(
a_{1},a_{2},\ldots,a_{k}\right)  $, $\left(  \alpha_{1},\alpha_{2}%
,\ldots,\alpha_{k}\right)  $, $s$ and $t$) shows that $\left(  \beta
_{p+1},\beta_{p+2},\ldots,\beta_{p+m}\right)  =q\rho_{t,s}q^{-1}$ (where we
now use (\ref{pf.prop.Invsles.b.3}) instead of (\ref{pf.prop.Invsles.b.2}),
and use $q=b_{1}b_{2}\cdots b_{p}$ instead of $q=a_{1}a_{2}\cdots a_{p}$).

Now, recall that the word $q\rho_{t,s}q^{-1}$ is the reversal of the word
$q\rho_{s,t}q^{-1}$. Since \newline$\left(  \alpha_{p+1},\alpha_{p+2}%
,\ldots,\alpha_{p+m}\right)  =q\rho_{s,t}q^{-1}$ and $\left(  \beta
_{p+1},\beta_{p+2},\ldots,\beta_{p+m}\right)  =q\rho_{t,s}q^{-1}$, this means
that the word $\left(  \beta_{p+1},\beta_{p+2},\ldots,\beta_{p+m}\right)  $ is
the reversal of $\left(  \alpha_{p+1},\alpha_{p+2},\ldots,\alpha_{p+m}\right)
$. This proves Claim 3.

\textit{Proof of Claim 4:} Since $m=m_{s,t}$, we have $\underbrace{stst\cdots
}_{m\text{ letters}}=\underbrace{tsts\cdots}_{m\text{ letters}}$ (this is one
of the braid relations of our Coxeter group). Let us set
$x=\underbrace{stst\cdots}_{m\text{ letters}}=\underbrace{tsts\cdots}_{m\text{
letters}}$. Now, (\ref{pf.prop.Invsles.b.2}) yields $a_{p+1}a_{p+2}\cdots
a_{p+m}=\underbrace{stst\cdots}_{m\text{ letters}}=x$. Similarly, from
(\ref{pf.prop.Invsles.b.3}), we obtain $b_{p+1}b_{p+2}\cdots b_{p+m}=x$.

Let $i\in\left\{  p+m+1,p+m+2,\ldots,k\right\}  $. Thus,%
\begin{align*}
a_{1}a_{2}\cdots a_{i-1}  &  =\underbrace{\left(  a_{1}a_{2}\cdots
a_{p}\right)  }_{=q}\underbrace{\left(  a_{p+1}a_{p+2}\cdots a_{p+m}\right)
}_{=x}\underbrace{\left(  a_{p+m+1}a_{p+m+2}\cdots a_{i-1}\right)
}_{\substack{=b_{p+m+1}b_{p+m+2}\cdots b_{i-1}\\\text{(by
(\ref{pf.prop.Invsles.b.4}))}}}\\
&  =qx\left(  b_{p+m+1}b_{p+m+2}\cdots b_{i-1}\right)  .
\end{align*}
Comparing this with
\begin{align*}
b_{1}b_{2}\cdots b_{i-1}  &  =\underbrace{\left(  b_{1}b_{2}\cdots
b_{p}\right)  }_{=q}\underbrace{\left(  b_{p+1}b_{p+2}\cdots b_{p+m}\right)
}_{=x}\left(  b_{p+m+1}b_{p+m+2}\cdots b_{i-1}\right) \\
&  =qx\left(  b_{p+m+1}b_{p+m+2}\cdots b_{i-1}\right)  ,
\end{align*}
we obtain $a_{1}a_{2}\cdots a_{i-1}=b_{1}b_{2}\cdots b_{i-1}$. Also,
$a_{i}=b_{i}$ (by (\ref{pf.prop.Invsles.b.4})). Now,
(\ref{pf.prop.Invsles.b.alpha}) becomes%
\begin{align*}
\alpha_{i}  &  =\left(  \underbrace{a_{1}a_{2}\cdots a_{i-1}}_{=b_{1}%
b_{2}\cdots b_{i-1}}\right)  \underbrace{a_{i}}_{=b_{i}}\left(
\underbrace{a_{1}a_{2}\cdots a_{i-1}}_{=b_{1}b_{2}\cdots b_{i-1}}\right)
^{-1}=\left(  b_{1}b_{2}\cdots b_{i-1}\right)  b_{i}\left(  b_{1}b_{2}\cdots
b_{i-1}\right)  ^{-1}\\
&  =\beta_{i}\ \ \ \ \ \ \ \ \ \ \left(  \text{by
(\ref{pf.prop.Invsles.b.beta})}\right)  .
\end{align*}
This proves Claim 4.

Hence, all four claims are proven, and the proof of Proposition
\ref{prop.Invsles} \textbf{(b)} is complete.
\end{proof}

The following fact is rather easy (but will be proven in detail in the next section):

\begin{proposition}
\label{prop.has}Let $w\in W$. Let $s$ and $t$ be two distinct elements of $T$
such that $m_{s,t}<\infty$. Let $\overrightarrow{a}$ be a reduced expression
for $w$. Then:

\begin{enumerate}
\item[\textbf{(a)}] The word $\rho_{s,t}$ appears as a subword of
$\operatorname*{Invs}\overrightarrow{a}$ at most one time.

\item[\textbf{(b)}] The words $\rho_{s,t}$ and $\rho_{t,s}$ cannot both appear
as subwords of $\operatorname*{Invs}\overrightarrow{a}$.
\end{enumerate}
\end{proposition}

\begin{proof}
[Proof of Proposition \ref{prop.has}.]\textbf{(a)} This follows from the fact
that the word $\rho_{s,t}$ has length $m_{s,t}\geq2>0$, and from Proposition
\ref{prop.Invsles} \textbf{(a)}. \medskip

\textbf{(b)} Assume the contrary. Then, both words $\rho_{s,t}$ and
$\rho_{t,s}$ appear as a subword of $\operatorname*{Invs}\overrightarrow{a}$.
By Proposition \ref{prop.rhost} \textbf{(b)}, this means that both the word
$\rho_{s,t}$ and its reversal appear as a subword of $\operatorname*{Invs}%
\overrightarrow{a}$. Since the word $\rho_{s,t}$ has length $m_{s,t}\geq2$,
this means that at least one letter of $\rho_{s,t}$ appears twice in
$\operatorname*{Invs}\overrightarrow{a}$. This contradicts Proposition
\ref{prop.Invsles} \textbf{(a)}. This contradiction concludes our proof.
\end{proof}

\section{The set $\mathfrak{N}$ and subwords of inversion words}

We now let $\mathfrak{N}$ denote the subset $\bigcup\limits_{x\in
W}x\mathfrak{M}x^{-1}$ of $T\times T$. Clearly, $\mathfrak{M}\subseteq
\mathfrak{N}$. Moreover, for every $\left(  s,t\right)  \in\mathfrak{N}$, we
have $s\neq t$ and $m_{s,t}<\infty$ (because $\left(  s,t\right)
\in\mathfrak{N}=\bigcup\limits_{x\in W}x\mathfrak{M}x^{-1}$, and because these
properties are preserved by conjugation). Thus, for every $\left(  s,t\right)
\in\mathfrak{N}$, the word $\rho_{s,t}$ is well-defined and has exactly
$m_{s,t}$ entries.

We define a binary relation $\approx$ on $\mathfrak{N}$ by%
\[
\left(  \left(  s,t\right)  \approx\left(  s^{\prime},t^{\prime}\right)
\ \Longleftrightarrow\ \text{there exists a }q\in W\text{ such that }%
qsq^{-1}=s^{\prime}\text{ and }qtq^{-1}=t^{\prime}\right)  .
\]
It is clear that this relation $\approx$ is an equivalence relation; it thus
gives rise to a quotient set $\mathfrak{N}/\approx$. For every pair
$P\in\mathfrak{N}$, we denote by $\left[  \left[  P\right]  \right]  $ the
equivalence class of $P$ with respect to this relation $\approx$.

The relation $\sim$ on $\mathfrak{M}$ is the restriction of the relation
$\approx$ to $\mathfrak{M}$. Hence, every equivalence class $c$ with respect
to $\sim$ is a subset of an equivalence class with respect to $\approx$. We
denote the latter equivalence class by $c_{\mathfrak{N}}$. Thus, $\left[
P\right]  _{\mathfrak{N}}=\left[  \left[  P\right]  \right]  $ for every
$P\in\mathfrak{M}$.

We notice that the set $\mathfrak{N}$ is invariant under switching the two
elements of a pair (i.e., for every $\left(  u,v\right)  \in\mathfrak{N}$, we
have $\left(  v,u\right)  \in\mathfrak{N}$). Moreover, the relation $\approx$
is preserved under switching the two elements of a pair (i.e., if $\left(
s,t\right)  \approx\left(  s^{\prime},t^{\prime}\right)  $, then $\left(
t,s\right)  \approx\left(  t^{\prime},s^{\prime}\right)  $). This shall be
tacitly used in the following proofs.

\begin{definition}
\label{def.has}Let $w\in W$. Let $\overrightarrow{a}$ be a reduced expression
for $w$.

\begin{enumerate}
\item[\textbf{(a)}] For any $\left(  s,t\right)  \in\mathfrak{N}$, we define
an element $\operatorname*{has}\nolimits_{s,t}\overrightarrow{a}\in\left\{
0,1\right\}  $ by%
\[
\operatorname*{has}\nolimits_{s,t}\overrightarrow{a}=%
\begin{cases}
1, & \text{if }\rho_{s,t}\text{ appears as a subword of }\operatorname*{Invs}%
\overrightarrow{a};\\
0, & \text{otherwise}%
\end{cases}
.
\]
(Keep in mind that we are speaking of subwords, not just factors, here.)

\item[\textbf{(b)}] Consider the free $\mathbb{Z}$-module $\mathbb{Z}\left[
\mathfrak{N}\right]  $ with basis $\mathfrak{N}$. We define an element
$\operatorname*{Has}\overrightarrow{a}\in\mathbb{Z}\left[  \mathfrak{N}%
\right]  $ by%
\[
\operatorname*{Has}\overrightarrow{a}=\sum_{\left(  s,t\right)  \in
\mathfrak{N}}\operatorname*{has}\nolimits_{s,t}\overrightarrow{a}\cdot\left(
s,t\right)
\]
(where the $\left(  s,t\right)  $ stands for the basis element $\left(
s,t\right)  \in\mathfrak{N}$ of $\mathbb{Z}\left[  \mathfrak{N}\right]  $).
This is well-defined, since only finitely many pairs $\left(  s,t\right)
\in\mathfrak{N}$ satisfy $\operatorname*{has}\nolimits_{s,t}\overrightarrow{a}%
\neq0$ (since $\operatorname*{Invs}\overrightarrow{a}$ has only finitely many
subwords, and all the $\rho_{s,t}$ are distinct).
\end{enumerate}
\end{definition}

We can now state the main result that we will use to prove Theorem
\ref{thm.BCL}:

\begin{theorem}
\label{thm.has}Let $w\in W$. Let $\left(  s,t\right)  \in\mathfrak{M}$. Let
$\overrightarrow{a}$ and $\overrightarrow{b}$ be two reduced expressions for
$w$ such that $\overrightarrow{b}$ is obtained from $\overrightarrow{a}$ by an
$\left(  s,t\right)  $-braid move.

Proposition \ref{prop.Invsles} \textbf{(b)} shows that there exists a $q\in W$
such that $\operatorname*{Invs}\overrightarrow{b}$ is obtained from
$\operatorname*{Invs}\overrightarrow{a}$ by replacing a particular factor of
the form $q\rho_{s,t}q^{-1}$ by its reversal. Consider this $q$. Set
$s^{\prime}=qsq^{-1}$ and $t^{\prime}=qtq^{-1}$; thus, $s^{\prime}$ and
$t^{\prime}$ are reflections and satisfy $m_{s^{\prime},t^{\prime}}%
=m_{s,t}<\infty$. Also, the definitions of $s^{\prime}$ and $t^{\prime}$ yield
$\left(  s^{\prime},t^{\prime}\right)  =q\underbrace{\left(  s,t\right)
}_{\in\mathfrak{M}}q^{-1}\in q\mathfrak{M}q^{-1}\subseteq\mathfrak{N}$.
Similarly, $\left(  t^{\prime},s^{\prime}\right)  \in\mathfrak{N}$ (since
$\left(  t,s\right)  \in\mathfrak{M}$).

Now, we have%
\begin{equation}
\operatorname*{Has}\overrightarrow{b}=\operatorname*{Has}\overrightarrow{a}%
-\left(  s^{\prime},t^{\prime}\right)  +\left(  t^{\prime},s^{\prime}\right)
. \label{eq.thm.has.a}%
\end{equation}

\end{theorem}

Before we prove Theorem \ref{thm.has}, we first show two lemmas. The first one
is a crucial property of dihedral subgroups in our Coxeter group:

\begin{lemma}
\label{lem.dihindih}Let $\left(  s,t\right)  \in\mathfrak{M}$ and $\left(
u,v\right)  \in\mathfrak{N}$. Let $q\in W$. Assume that $u\in qD_{s,t}q^{-1}$
and $v\in qD_{s,t}q^{-1}$. Then, $m_{s,t}=m_{u,v}$.
\end{lemma}

\begin{proof}
[Proof of Lemma \ref{lem.dihindih}.]\textit{Claim 1:} Lemma \ref{lem.dihindih}
holds in the case when $\left(  u,v\right)  \in\mathfrak{M}$.

\textit{Proof.} Assume that $\left(  u,v\right)  \in\mathfrak{M}$. Thus,
$u,v\in S$. Let $I$ be the subset $\left\{  s,t\right\}  $ of $S$. We shall
use the notations of \cite[\S 9]{Lusztig-Hecke}. In particular, $l\left(
r\right)  $ denotes the length of any element $r\in W$.

We have $W_{I}=D_{s,t}$. Consider the coset $W_{I}q^{-1}$ of $W_{I}$. From
\cite[Lemma 9.7 (a)]{Lusztig-Hecke} (applied to $a=q^{-1}$), we know that this
coset $W_{I}q^{-1}$ has a unique element of minimal length. Let $w$ be this
element. Thus, $w\in W_{I}q^{-1}$, so that $W_{I}w=W_{I}q^{-1}$. Now,
\[
\underbrace{q}_{=\left(  q^{-1}\right)  ^{-1}}\underbrace{W_{I}}_{=\left(
W_{I}\right)  ^{-1}}=\left(  q^{-1}\right)  ^{-1}\left(  W_{I}\right)
^{-1}=\left(  \underbrace{W_{I}q^{-1}}_{=W_{I}w}\right)  ^{-1}=\left(
W_{I}w\right)  ^{-1}=w^{-1}W_{I}.
\]

Let $u^{\prime}=wuw^{-1}$ and $v^{\prime}=wvw^{-1}$.

We have $u\in q\underbrace{D_{s,t}}_{=W_{I}}q^{-1}=q\underbrace{W_{I}q^{-1}%
}_{=W_{I}w}=\underbrace{qW_{I}}_{=w^{-1}W_{I}}w=w^{-1}W_{I}w$. In other words,
$wuw^{-1}\in W_{I}$. In other words, $u^{\prime}\in W_{I}$ (since $u^{\prime
}=wuw^{-1}$). Similarly, $v^{\prime}\in W_{I}$.

We have $u^{\prime}=wuw^{-1}$, hence $u^{\prime}w=wu$. But \cite[Lemma 9.7
(b)]{Lusztig-Hecke} (applied to $a=q^{-1}$ and $y=u^{\prime}$) shows that
$l\left(  u^{\prime}w\right)  =l\left(  u^{\prime}\right)  +l\left(  w\right)
$. Hence,
\[
l\left(  u^{\prime}\right)  +l\left(  w\right)  =l\left(
\underbrace{u^{\prime}w}_{=wu}\right)  =l\left(  wu\right)  =l\left(
w\right)  \pm1\ \ \ \ \ \ \ \ \ \ \left(  \text{since }u\in S\right)  .
\]
Subtracting $l\left(  w\right)  $ from this equality, we obtain $l\left(
u^{\prime}\right)  =\pm1$, and thus $l\left(  u^{\prime}\right)  =1$, so that
$u^{\prime}\in S$. Combined with $u^{\prime}\in W_{I}$, this shows that
$u^{\prime}\in S\cap W_{I}=I$. Similarly, $v^{\prime}\in I$.

We have $u\neq v$ (since $\left(  u,v\right)  \in\mathfrak{N}$), thus
$wuw^{-1}\neq wvw^{-1}$, thus $u^{\prime}=wuw^{-1}\neq wvw^{-1}=v^{\prime}$.
Thus, $u^{\prime}$ and $v^{\prime}$ are two distinct elements of the
two-element set $I=\left\{  s,t\right\}  $. Hence, either $\left(  u^{\prime
},v^{\prime}\right)  =\left(  s,t\right)  $ or $\left(  u^{\prime},v^{\prime
}\right)  =\left(  t,s\right)  $. In either of these two cases, we have
$m_{u^{\prime},v^{\prime}}=m_{s,t}$. But since $u^{\prime}=wuw^{-1}$ and
$v^{\prime}=wvw^{-1}$, we have $m_{u^{\prime},v^{\prime}}=m_{u,v}$. Hence,
$m_{s,t}=m_{u^{\prime},v^{\prime}}=m_{u,v}$. This proves Claim 1.

\textit{Claim 2:} Lemma \ref{lem.dihindih} holds in the general case.

\textit{Proof.} Consider the general case. We have $\left(  u,v\right)
\in\mathfrak{N}=\bigcup_{x\in W}x\mathfrak{M}x^{-1}$. Thus, there exists some
$x\in W$ such that $\left(  u,v\right)  \in x\mathfrak{M}x^{-1}$. Consider
this $x$. From $\left(  u,v\right)  \in x\mathfrak{M}x^{-1}$, we obtain
$x^{-1}\left(  u,v\right)  x\in\mathfrak{M}$. In other words, $\left(
x^{-1}ux,x^{-1}vx\right)  \in\mathfrak{M}$. Moreover,
\[
x^{-1}\underbrace{u}_{\in qD_{s,t}q^{-1}}x\in x^{-1}qD_{s,t}\underbrace{q^{-1}%
x}_{=\left(  x^{-1}q\right)  ^{-1}}=x^{-1}qD_{s,t}\left(  x^{-1}q\right)
^{-1},
\]
and similarly $x^{-1}vx\in x^{-1}qD_{s,t}\left(  x^{-1}q\right)  ^{-1}$.
Hence, Claim 1 (applied to $\left(  x^{-1}ux,x^{-1}vx\right)  $ and $x^{-1}q$
instead of $\left(  u,v\right)  $ and $q$) shows that $m_{s,t}=m_{x^{-1}%
ux,x^{-1}vx}=m_{u,v}$. This proves Claim 2, and thus proves Lemma
\ref{lem.dihindih}.
\end{proof}

Next comes another lemma, bordering on the trivial:

\begin{lemma}
\label{lem.GandH}Let $G$ be a group. Let $H$ be a subgroup of $G$. Let $u\in
G$, $v\in G$ and $g\in\mathbb{Z}$. Assume that $\left(  uv\right)  ^{g-1}u\in
H$ and $\left(  uv\right)  ^{g}u\in H$. Then, $u\in H$ and $v\in H$.
\end{lemma}

\begin{proof}
[Proof of Lemma \ref{lem.GandH}.]We have $\underbrace{\left(  \left(
uv\right)  ^{g}u\right)  }_{\in H}\left(  \underbrace{\left(  uv\right)
^{g-1}u}_{\in H}\right)  ^{-1}\in HH^{-1}\subseteq H$ (since $H$ is a subgroup
of $G$). Since
\[
\left(  \left(  uv\right)  ^{g}u\right)  \underbrace{\left(  \left(
uv\right)  ^{g-1}u\right)  ^{-1}}_{=u^{-1}\left(  \left(  uv\right)
^{g-1}\right)  ^{-1}}=\left(  uv\right)  ^{g}\underbrace{uu^{-1}}_{=1}\left(
\left(  uv\right)  ^{g-1}\right)  ^{-1}=\left(  uv\right)  ^{g}\left(  \left(
uv\right)  ^{g-1}\right)  ^{-1}=uv,
\]
this rewrites as $uv\in H$. However, $\left(  uv\right)  ^{-g}\left(
uv\right)  ^{g}u=u$, so that
\[
u=\left(  \underbrace{uv}_{\in H}\right)  ^{-g}\underbrace{\left(  uv\right)
^{g}u}_{\in H}\in H^{-g}H\subseteq H
\]
(since $H$ is a subgroup of $G$). Now, both $u$ and $uv$ belong to the
subgroup $H$ of $G$. Thus, so does $u^{-1}\left(  uv\right)  $. In other
words, $u^{-1}\left(  uv\right)  \in H$, so that $v=u^{-1}\left(  uv\right)
\in H$. This completes the proof of Lemma \ref{lem.GandH}.
\end{proof}

\begin{proof}
[Proof of Theorem \ref{thm.has}.]Conjugation by $q$ (that is, the map
$W\rightarrow W,\ x\mapsto qxq^{-1}$) is a group endomorphism of $W$. Hence,
for every $i\in\mathbb{N}$, we have%
\begin{equation}
q\left(  st\right)  ^{i}sq^{-1}=\left(  \underbrace{\left(  qsq^{-1}\right)
}_{=s^{\prime}}\left(  \underbrace{qtq^{-1}}_{=t^{\prime}}\right)  \right)
^{i}\underbrace{\left(  qsq^{-1}\right)  }_{=s^{\prime}}=\left(  s^{\prime
}t^{\prime}\right)  ^{i}s^{\prime}. \label{pf.thm.has.qconj}%
\end{equation}

Let $m=m_{s,t}$. We have%
\[
\rho_{s,t}=\left(  \left(  st\right)  ^{0}s,\left(  st\right)  ^{1}%
s,\ldots,\left(  st\right)  ^{m_{s,t}-1}s\right)  =\left(  \left(  st\right)
^{0}s,\left(  st\right)  ^{1}s,\ldots,\left(  st\right)  ^{m-1}s\right)
\]
(since $m_{s,t}=m$) and thus%
\begin{align*}
q\rho_{s,t}q^{-1}  &  =q\left(  \left(  st\right)  ^{0}s,\left(  st\right)
^{1}s,\ldots,\left(  st\right)  ^{m-1}s\right)  q^{-1}\\
&  =\left(  q\left(  st\right)  ^{0}sq^{-1},q\left(  st\right)  ^{1}%
sq^{-1},\ldots,q\left(  st\right)  ^{m-1}sq^{-1}\right) \\
&  =\left(  \left(  s^{\prime}t^{\prime}\right)  ^{0}s^{\prime},\left(
s^{\prime}t^{\prime}\right)  ^{1}s^{\prime},\ldots,\left(  s^{\prime}%
t^{\prime}\right)  ^{m-1}s^{\prime}\right) \\
&  \ \ \ \ \ \ \ \ \ \ \left(
\begin{array}
[c]{c}%
\text{since every }i\in\left\{  0,1,\ldots,m-1\right\}  \text{ satisfies}\\
q\left(  st\right)  ^{i}sq^{-1}=\left(  s^{\prime}t^{\prime}\right)
^{i}s^{\prime}\text{ (by (\ref{pf.thm.has.qconj}))}%
\end{array}
\right) \\
&  =\left(  \left(  s^{\prime}t^{\prime}\right)  ^{0}s^{\prime},\left(
s^{\prime}t^{\prime}\right)  ^{1}s^{\prime},\ldots,\left(  s^{\prime}%
t^{\prime}\right)  ^{m_{s^{\prime},t^{\prime}}-1}s^{\prime}\right)
\ \ \ \ \ \ \ \ \ \ \left(  \text{since }m=m_{s,t}=m_{s^{\prime},t^{\prime}%
}\right) \\
&  =\rho_{s^{\prime},t^{\prime}}\ \ \ \ \ \ \ \ \ \ \left(  \text{by the
definition of }\rho_{s^{\prime},t^{\prime}}\right)  .
\end{align*}

The word $\overrightarrow{b}$ is obtained from $\overrightarrow{a}$ by an
$\left(  s,t\right)  $-braid move. Hence, the word $\overrightarrow{a}$ can be
obtained from $\overrightarrow{b}$ by a $\left(  t,s\right)  $-braid move.

From $\left(  s^{\prime},t^{\prime}\right)  \in\mathfrak{N}$, we obtain
$s^{\prime}\neq t^{\prime}$. Hence, $\left(  s^{\prime},t^{\prime}\right)
\neq\left(  t^{\prime},s^{\prime}\right)  $.

From $s^{\prime}=qsq^{-1}$ and $t^{\prime}=qtq^{-1}$, we obtain $D_{s^{\prime
},t^{\prime}}=qD_{s,t}q^{-1}$ (since conjugation by $q$ is a group
endomorphism of $W$).

Proposition \ref{prop.rhost} \textbf{(c)} shows that the word $q\rho
_{t,s}q^{-1}$ is the reversal of the word $q\rho_{s,t}q^{-1}$. Hence, the word
$q\rho_{s,t}q^{-1}$ is the reversal of the word $q\rho_{t,s}q^{-1}$.

Recall that $\operatorname*{Invs}\overrightarrow{b}$ is obtained from
$\operatorname*{Invs}\overrightarrow{a}$ by replacing a particular factor of
the form $q\rho_{s,t}q^{-1}$ by its reversal. Since this latter reversal is
$q\rho_{t,s}q^{-1}$ (as we have previously seen), this shows that
$\operatorname*{Invs}\overrightarrow{b}$ has a factor of $q\rho_{t,s}q^{-1}$
in the place where the word $\operatorname*{Invs}\overrightarrow{a}$ had the
factor $q\rho_{s,t}q^{-1}$. Hence, $\operatorname*{Invs}\overrightarrow{a}$
can, in turn, be obtained from $\operatorname*{Invs}\overrightarrow{b}$ by
replacing a particular factor of the form $q\rho_{t,s}q^{-1}$ by its reversal
(since the reversal of $q\rho_{t,s}q^{-1}$ is $q\rho_{s,t}q^{-1}$). Thus, our
situation is symmetric with respect to $s$ and $t$; more precisely, we wind up
in an analogous situation if we replace $s$, $t$, $\overrightarrow{a}$,
$\overrightarrow{b}$, $s^{\prime}$ and $t^{\prime}$ by $t$, $s$,
$\overrightarrow{b}$, $\overrightarrow{a}$, $t^{\prime}$ and $s^{\prime}$, respectively.

We shall prove the following claims:

\textit{Claim 1:} Let $\left(  u,v\right)  \in\mathfrak{N}$ be such that
$\left(  u,v\right)  \neq\left(  s^{\prime},t^{\prime}\right)  $ and $\left(
u,v\right)  \neq\left(  t^{\prime},s^{\prime}\right)  $. Then,
$\operatorname*{has}\nolimits_{u,v}\overrightarrow{b}=\operatorname*{has}%
\nolimits_{u,v}\overrightarrow{a}$.

\textit{Claim 2:} We have $\operatorname*{has}\nolimits_{s^{\prime},t^{\prime
}}\overrightarrow{b}=\operatorname*{has}\nolimits_{s^{\prime},t^{\prime}%
}\overrightarrow{a}-1$.

\textit{Claim 3:} We have $\operatorname*{has}\nolimits_{t^{\prime},s^{\prime
}}\overrightarrow{b}=\operatorname*{has}\nolimits_{t^{\prime},s^{\prime}%
}\overrightarrow{a}+1$.

\textit{Proof of Claim 1:} Assume the contrary. Thus, $\operatorname*{has}%
\nolimits_{u,v}\overrightarrow{b}\neq\operatorname*{has}\nolimits_{u,v}%
\overrightarrow{a}$. Hence, one of the numbers $\operatorname*{has}%
\nolimits_{u,v}\overrightarrow{b}$ and $\operatorname*{has}\nolimits_{u,v}%
\overrightarrow{a}$ equals $1$ and the other equals $0$ (since both
$\operatorname*{has}\nolimits_{u,v}\overrightarrow{b}$ and
$\operatorname*{has}\nolimits_{u,v}\overrightarrow{a}$ belong to $\left\{
0,1\right\}  $). Without loss of generality, we assume that
$\operatorname*{has}\nolimits_{u,v}\overrightarrow{a}=1$ and
$\operatorname*{has}\nolimits_{u,v}\overrightarrow{b}=0$ (because in the other
case, we can replace $s$, $t$, $\overrightarrow{a}$, $\overrightarrow{b}$,
$s^{\prime}$ and $t^{\prime}$ by $t$, $s$, $\overrightarrow{b}$,
$\overrightarrow{a}$, $t^{\prime}$ and $s^{\prime}$, respectively).

The elements $u$ and $v$ are two distinct reflections (since $\left(
u,v\right)  \in\mathfrak{N}$).

Write the tuple $\operatorname*{Invs}\overrightarrow{a}$ as $\left(
\alpha_{1},\alpha_{2},\ldots,\alpha_{k}\right)  $. The tuple
$\operatorname*{Invs}\overrightarrow{b}$ has the same length as
$\operatorname*{Invs}\overrightarrow{a}$, since $\operatorname*{Invs}%
\overrightarrow{b}$ is obtained from $\operatorname*{Invs}\overrightarrow{a}$
by replacing a particular factor of the form $q\rho_{s,t}q^{-1}$ by its
reversal. Hence, write the tuple $\operatorname*{Invs}\overrightarrow{b}$ as
$\left(  \beta_{1},\beta_{2},\ldots,\beta_{k}\right)  $.

From $\operatorname*{has}\nolimits_{u,v}\overrightarrow{a}=1$, we obtain that
$\rho_{u,v}$ appears as a subword of $\operatorname*{Invs}\overrightarrow{a}$.
In other words, $\rho_{u,v}=\left(  \alpha_{i_{1}},\alpha_{i_{2}}%
,\ldots,\alpha_{i_{f}}\right)  $ for some integers $i_{1},i_{2},\ldots,i_{f}$
satisfying $1\leq i_{1}<i_{2}<\cdots<i_{f}\leq k$. Consider these $i_{1}%
,i_{2},\ldots,i_{f}$. From $\operatorname*{has}\nolimits_{u,v}%
\overrightarrow{b}=0$, we conclude that $\rho_{u,v}$ does not appear as a
subword of $\operatorname*{Invs}\overrightarrow{b}$.

On the other hand, $\operatorname*{Invs}\overrightarrow{b}$ is obtained from
$\operatorname*{Invs}\overrightarrow{a}$ by replacing a particular factor of
the form $q\rho_{s,t}q^{-1}$ by its reversal. This factor has $m_{s,t}=m$
letters; thus, it has the form $\left(  \alpha_{p+1},\alpha_{p+2}%
,\ldots,\alpha_{p+m}\right)  $ for some $p\in\left\{  0,1,\ldots,k-m\right\}
$. Consider this $p$. Thus,%
\[
\left(  \alpha_{p+1},\alpha_{p+2},\ldots,\alpha_{p+m}\right)  =q\rho
_{s,t}q^{-1}=\left(  \left(  s^{\prime}t^{\prime}\right)  ^{0}s^{\prime
},\left(  s^{\prime}t^{\prime}\right)  ^{1}s^{\prime},\ldots,\left(
s^{\prime}t^{\prime}\right)  ^{m-1}s^{\prime}\right)  .
\]
In other words,%
\begin{equation}
\alpha_{p+i}=\left(  s^{\prime}t^{\prime}\right)  ^{i-1}s^{\prime
}\ \ \ \ \ \ \ \ \ \ \text{for every }i\in\left\{  1,2,\ldots,m\right\}  .
\label{pf.thm.has.c1.2}%
\end{equation}

We now summarize:

\begin{itemize}
\item The word $\rho_{u,v}$ appears as the subword $\left(  \alpha_{i_{1}%
},\alpha_{i_{2}},\ldots,\alpha_{i_{f}}\right)  $ of $\operatorname*{Invs}%
\overrightarrow{a}$, but does not appear as a subword of $\operatorname*{Invs}%
\overrightarrow{b}$.

\item The word $\operatorname*{Invs}\overrightarrow{b}$ is obtained from
$\operatorname*{Invs}\overrightarrow{a}$ by replacing the factor
\newline$\left(  \alpha_{p+1},\alpha_{p+2},\ldots,\alpha_{p+m}\right)  $ by
its reversal.
\end{itemize}

Thus, replacing the factor $\left(  \alpha_{p+1},\alpha_{p+2},\ldots
,\alpha_{p+m}\right)  $ in $\operatorname*{Invs}\overrightarrow{a}$ by its
reversal must mess up the subword $\left(  \alpha_{i_{1}},\alpha_{i_{2}%
},\ldots,\alpha_{i_{f}}\right)  $ of $\operatorname*{Invs}\overrightarrow{a}$
badly enough that it no longer appears as a subword (not even in different
positions). This can only happen if at least two of the integers $i_{1}%
,i_{2},\ldots,i_{f}$ lie in the interval $\left\{  p+1,p+2,\ldots,p+m\right\}
$.

Hence, at least two of the integers $i_{1},i_{2},\ldots,i_{f}$ lie in the
interval $\left\{  p+1,p+2,\ldots,p+m\right\}  $. In particular, there must be
a $g\in\left\{  1,2,\ldots,f-1\right\}  $ such that the integers $i_{g}$ and
$i_{g+1}$ lie in the interval $\left\{  p+1,p+2,\ldots,p+m\right\}  $ (since
$i_{1}<i_{2}<\cdots<i_{f}$). Consider this $g$.

We have $i_{g}\in\left\{  p+1,p+2,\ldots,p+m\right\}  $. In other words,
$i_{g}=p+r_{g}$ for some $r_{g}\in\left\{  1,2,\ldots,m\right\}  $. Consider
this $r_{g}$.

We have $i_{g+1}\in\left\{  p+1,p+2,\ldots,p+m\right\}  $. In other words,
$i_{g+1}=p+r_{g+1}$ for some $r_{g+1}\in\left\{  1,2,\ldots,m\right\}  $.
Consider this $r_{g+1}$.

We have $\left(  \alpha_{i_{1}},\alpha_{i_{2}},\ldots,\alpha_{i_{f}}\right)
=\rho_{u,v}=\left(  \left(  uv\right)  ^{0}u,\left(  uv\right)  ^{1}%
u,\ldots,\left(  uv\right)  ^{m_{u,v}-1}u\right)  $ (by the definition of
$\rho_{u,v}$). Hence, $\alpha_{i_{g}}=\left(  uv\right)  ^{g-1}u$ and
$\alpha_{i_{g+1}}=\left(  uv\right)  ^{g}u$. Now,%
\begin{align*}
\left(  uv\right)  ^{g-1}u  &  =\alpha_{i_{g}}=\alpha_{p+r_{g}}%
\ \ \ \ \ \ \ \ \ \ \left(  \text{since }i_{g}=p+r_{g}\right) \\
&  =\left(  s^{\prime}t^{\prime}\right)  ^{r_{g}-1}s^{\prime}%
\ \ \ \ \ \ \ \ \ \ \left(  \text{by (\ref{pf.thm.has.c1.2}), applied to
}i=r_{g}\right) \\
&  \in D_{s^{\prime},t^{\prime}}%
\end{align*}
and%
\begin{align*}
\left(  uv\right)  ^{g}u  &  =\alpha_{i_{g+1}}=\alpha_{p+r_{g+1}%
}\ \ \ \ \ \ \ \ \ \ \left(  \text{since }i_{g+1}=p+r_{g+1}\right) \\
&  =\left(  s^{\prime}t^{\prime}\right)  ^{r_{g+1}-1}s^{\prime}%
\ \ \ \ \ \ \ \ \ \ \left(  \text{by (\ref{pf.thm.has.c1.2}), applied to
}i=r_{g+1}\right) \\
&  \in D_{s^{\prime},t^{\prime}}.
\end{align*}
Hence, Lemma \ref{lem.GandH} (applied to $G=W$ and $H=D_{s^{\prime},t^{\prime
}}$) yields $u\in D_{s^{\prime},t^{\prime}}$ and $v\in D_{s^{\prime}%
,t^{\prime}}$.

Furthermore, we have%
\[
\alpha_{i_{1}}=u\ \ \ \ \ \ \ \ \ \ \text{and}\ \ \ \ \ \ \ \ \ \ \alpha
_{i_{f}}=v
\]
\footnote{\textit{Proof.} From $\left(  \alpha_{i_{1}},\alpha_{i_{2}}%
,\ldots,\alpha_{i_{f}}\right)  =\left(  \left(  uv\right)  ^{0}u,\left(
uv\right)  ^{1}u,\ldots,\left(  uv\right)  ^{m_{u,v}-1}u\right)  $, we obtain
$\alpha_{i_{1}}=\underbrace{\left(  uv\right)  ^{0}}_{=1}u=u$.
\par
We have $\left(  uv\right)  ^{m_{u,v}}=1$, and thus $\left(  uv\right)
^{m_{u,v}-1}=\left(  uv\right)  ^{-1}=v^{-1}u^{-1}$.
\par
From $\left(  \alpha_{i_{1}},\alpha_{i_{2}},\ldots,\alpha_{i_{f}}\right)
=\left(  \left(  uv\right)  ^{0}u,\left(  uv\right)  ^{1}u,\ldots,\left(
uv\right)  ^{m_{u,v}-1}u\right)  $, we obtain $\alpha_{i_{f}}%
=\underbrace{\left(  uv\right)  ^{m_{u,v}-1}}_{=v^{-1}u^{-1}}u=v^{-1}%
u^{-1}u=v^{-1}=v$ (since $v$ is a reflection), qed.}.

Now, we have $i_{1}\in\left\{  p+1,p+2,\ldots,p+m\right\}  $ (by a simple
argument\footnote{\textit{Proof.} The element $u$ is a reflection and lies in
$D_{s^{\prime},t^{\prime}}$. Hence, Proposition \ref{prop.rhost} \textbf{(a)}
(applied to $s^{\prime}$ and $t^{\prime}$ instead of $s$ and $t$) shows that
the word $\rho_{s^{\prime},t^{\prime}}$ contains $u$. Since $\rho_{s^{\prime
},t^{\prime}}=q\rho_{s,t}q^{-1}=\left(  \alpha_{p+1},\alpha_{p+2}%
,\ldots,\alpha_{p+m}\right)  $, this shows that the word $\left(  \alpha
_{p+1},\alpha_{p+2},\ldots,\alpha_{p+m}\right)  $ contains $u$. In other
words, $u=\alpha_{M}$ for some $M\in\left\{  p+1,p+2,\ldots,p+m\right\}  $.
Consider this $M$.
\par
But Proposition \ref{prop.Invsles} \textbf{(a)} shows that all entries of the
tuple $\operatorname*{Invs}\overrightarrow{a}$ are distinct. In other words,
the elements $\alpha_{1},\alpha_{2},\ldots,\alpha_{k}$ are pairwise distinct
(since those are the entries of $\operatorname*{Invs}\overrightarrow{a}$).
Hence, from $\alpha_{i_{1}}=u=\alpha_{M}$, we obtain $i_{1}=M\in\left\{
p+1,p+2,\ldots,p+m\right\}  $. Qed.}) and $i_{f}\in\left\{  p+1,p+2,\ldots
,p+m\right\}  $ (by a similar argument, with $v$ occasionally replacing $u$).
Thus, all of the integers $i_{1},i_{2},\ldots,i_{f}$ belong to $\left\{
p+1,p+2,\ldots,p+m\right\}  $ (since $i_{1}<i_{2}<\cdots<i_{f}$).

Now, recall that $f$ is the length of the word $\rho_{u,v}$ (since $\rho
_{u,v}=\left(  \alpha_{i_{1}},\alpha_{i_{2}},\ldots,\alpha_{i_{f}}\right)  $),
and thus equals $m_{u,v}$. Thus, $f=m_{u,v}$.

But $u\in D_{s^{\prime},t^{\prime}}=qD_{s,t}q^{-1}$ and $v\in D_{s^{\prime
},t^{\prime}}=qD_{s,t}q^{-1}$. Hence, Lemma \ref{lem.dihindih} yields
$m_{s,t}=m_{u,v}$. Since $m=m_{s,t}$ and $f=m_{u,v}$, this rewrites as $m=f$.

Recall that all of the integers $i_{1},i_{2},\ldots,i_{f}$ belong to $\left\{
p+1,p+2,\ldots,p+m\right\}  $. Since $i_{1}<i_{2}<\cdots<i_{f}$ and $f=m$,
these integers $i_{1},i_{2},\ldots,i_{f}$ form a strictly increasing sequence
of length $m$. Thus, $\left(  i_{1},i_{2},\ldots,i_{f}\right)  $ is a strictly
increasing sequence of length $m$ whose entries belong to $\left\{
p+1,p+2,\ldots,p+m\right\}  $. But the only such sequence is $\left(
p+1,p+2,\ldots,p+m\right)  $ (because the set $\left\{  p+1,p+2,\ldots
,p+m\right\}  $ has only $m$ elements). Thus, $\left(  i_{1},i_{2}%
,\ldots,i_{f}\right)  =\left(  p+1,p+2,\ldots,p+m\right)  $. In particular,
$i_{1}=p+1$ and $i_{f}=p+m$.

Now, $\alpha_{i_{1}}=u$, so that%
\begin{align*}
u  &  =\alpha_{i_{1}}=\alpha_{p+1}\ \ \ \ \ \ \ \ \ \ \left(  \text{since
}i_{1}=p+1\right) \\
&  =\underbrace{\left(  s^{\prime}t^{\prime}\right)  ^{1-1}}_{=1}s^{\prime
}\ \ \ \ \ \ \ \ \ \ \left(  \text{by (\ref{pf.thm.has.c1.2}), applied to
}i=1\right) \\
&  =s^{\prime}.
\end{align*}
Also, $\alpha_{i_{f}}=v$, so that%
\begin{align*}
v  &  =\alpha_{i_{f}}=\alpha_{p+m}\ \ \ \ \ \ \ \ \ \ \left(  \text{since
}i_{f}=p+m\right) \\
&  =\underbrace{\left(  s^{\prime}t^{\prime}\right)  ^{m-1}}%
_{\substack{=\left(  s^{\prime}t^{\prime}\right)  ^{-1}\\\text{(since }\left(
s^{\prime}t^{\prime}\right)  ^{m}=1\\\text{(since }m=m_{s,t}=m_{s^{\prime
},t^{\prime}}\text{))}}}s^{\prime}\ \ \ \ \ \ \ \ \ \ \left(  \text{by
(\ref{pf.thm.has.c1.2}), applied to }i=m\right) \\
&  =\left(  s^{\prime}t^{\prime}\right)  ^{-1}s^{\prime}=t^{\prime}.
\end{align*}
Combined with $u=s^{\prime}$, this yields $\left(  u,v\right)  =\left(
s^{\prime},t^{\prime}\right)  $, which contradicts $\left(  u,v\right)
\neq\left(  s^{\prime},t^{\prime}\right)  $. This contradiction proves that
our assumption was wrong. Claim 1 is proven.

\textit{Proof of Claim 2:} The word $\operatorname*{Invs}\overrightarrow{b}$
is obtained from $\operatorname*{Invs}\overrightarrow{a}$ by replacing a
particular factor of the form $q\rho_{s,t}q^{-1}$ by its reversal. Thus, the
word $\operatorname*{Invs}\overrightarrow{a}$ has a factor of the form
$q\rho_{s,t}q^{-1}$. Since $q\rho_{s,t}q^{-1}=\rho_{s^{\prime},t^{\prime}}$,
this means that the word $\operatorname*{Invs}\overrightarrow{a}$ has a factor
of the form $\rho_{s^{\prime},t^{\prime}}$. Consequently, the word
$\operatorname*{Invs}\overrightarrow{a}$ has a subword of the form
$\rho_{s^{\prime},t^{\prime}}$. In other words, $\operatorname*{has}%
\nolimits_{s^{\prime},t^{\prime}}\overrightarrow{a}=1$.

The same argument (applied to $t$, $s$, $\overrightarrow{b}$,
$\overrightarrow{a}$, $t^{\prime}$ and $s^{\prime}$ instead of $s$, $t$,
$\overrightarrow{a}$, $\overrightarrow{b}$, $s^{\prime}$ and $t^{\prime}$)
shows that $\operatorname*{has}\nolimits_{t^{\prime},s^{\prime}}%
\overrightarrow{b}=1$. In other words, the word $\operatorname*{Invs}%
\overrightarrow{b}$ has a subword of the form $\rho_{t^{\prime},s^{\prime}}$.
Hence, the word $\operatorname*{Invs}\overrightarrow{b}$ has no subword of the
form $\rho_{s^{\prime},t^{\prime}}$ (because Proposition \ref{prop.has}
\textbf{(b)} (applied to $\overrightarrow{b}$, $s^{\prime}$ and $t^{\prime}$
instead of $\overrightarrow{a}$, $s$ and $t$) shows that the words
$\rho_{s^{\prime},t^{\prime}}$ and $\rho_{t^{\prime},s^{\prime}}$ cannot both
appear as subwords of $\operatorname*{Invs}\overrightarrow{b}$). In other
words, $\operatorname*{has}\nolimits_{s^{\prime},t^{\prime}}\overrightarrow{b}%
=0$.

Combining this with $\operatorname*{has}\nolimits_{s^{\prime},t^{\prime}%
}\overrightarrow{a}=1$, we immediately obtain $\operatorname*{has}%
\nolimits_{s^{\prime},t^{\prime}}\overrightarrow{b}=\operatorname*{has}%
\nolimits_{s^{\prime},t^{\prime}}\overrightarrow{a}-1$. Thus, Claim 2 is proven.

\textit{Proof of Claim 3:} Applying Claim 2 to $t$, $s$, $\overrightarrow{b}$,
$\overrightarrow{a}$, $t^{\prime}$ and $s^{\prime}$ instead of $s$, $t$,
$\overrightarrow{a}$, $\overrightarrow{b}$, $s^{\prime}$ and $t^{\prime}$, we
obtain $\operatorname*{has}\nolimits_{t^{\prime},s^{\prime}}\overrightarrow{a}%
=\operatorname*{has}\nolimits_{t^{\prime},s^{\prime}}\overrightarrow{b}-1$. In
other words, $\operatorname*{has}\nolimits_{t^{\prime},s^{\prime}%
}\overrightarrow{b}=\operatorname*{has}\nolimits_{t^{\prime},s^{\prime}%
}\overrightarrow{a}+1$. This proves Claim 3.

Now, our goal is to prove that $\operatorname*{Has}\overrightarrow{b}%
=\operatorname*{Has}\overrightarrow{a}-\left(  s^{\prime},t^{\prime}\right)
+\left(  t^{\prime},s^{\prime}\right)  $. But the definition of
$\operatorname*{Has}\overrightarrow{b}$ yields%
\begin{align*}
&  \operatorname*{Has}\overrightarrow{b}\\
&  =\sum_{\left(  u,v\right)  \in\mathfrak{N}}\operatorname*{has}%
\nolimits_{u,v}\overrightarrow{b}\cdot\left(  u,v\right) \\
&  =\sum_{\substack{\left(  u,v\right)  \in\mathfrak{N};\\\left(  u,v\right)
\neq\left(  s^{\prime},t^{\prime}\right)  ;\\\left(  u,v\right)  \neq\left(
t^{\prime},s^{\prime}\right)  }}\underbrace{\operatorname*{has}\nolimits_{u,v}%
\overrightarrow{b}}_{\substack{=\operatorname*{has}\nolimits_{u,v}%
\overrightarrow{a}\\\text{(by Claim 1)}}}\cdot\left(  u,v\right)
+\underbrace{\operatorname*{has}\nolimits_{s^{\prime},t^{\prime}%
}\overrightarrow{b}}_{\substack{=\operatorname*{has}\nolimits_{s^{\prime
},t^{\prime}}\overrightarrow{a}-1\\\text{(by Claim 2)}}}\cdot\left(
s^{\prime},t^{\prime}\right)  +\underbrace{\operatorname*{has}%
\nolimits_{t^{\prime},s^{\prime}}\overrightarrow{b}}%
_{\substack{=\operatorname*{has}\nolimits_{t^{\prime},s^{\prime}%
}\overrightarrow{a}+1\\\text{(by Claim 3)}}}\cdot\left(  t^{\prime},s^{\prime
}\right) \\
&  \ \ \ \ \ \ \ \ \ \ \left(  \text{since }\left(  s^{\prime},t^{\prime
}\right)  \neq\left(  t^{\prime},s^{\prime}\right)  \right) \\
&  =\sum_{\substack{\left(  u,v\right)  \in\mathfrak{N};\\\left(  u,v\right)
\neq\left(  s^{\prime},t^{\prime}\right)  ;\\\left(  u,v\right)  \neq\left(
t^{\prime},s^{\prime}\right)  }}\operatorname*{has}\nolimits_{u,v}%
\overrightarrow{a}\cdot\left(  u,v\right)  +\left(  \operatorname*{has}%
\nolimits_{s^{\prime},t^{\prime}}\overrightarrow{a}-1\right)  \cdot\left(
s^{\prime},t^{\prime}\right)  +\left(  \operatorname*{has}\nolimits_{t^{\prime
},s^{\prime}}\overrightarrow{a}+1\right)  \cdot\left(  t^{\prime},s^{\prime
}\right) \\
&  =\sum_{\substack{\left(  u,v\right)  \in\mathfrak{N};\\\left(  u,v\right)
\neq\left(  s^{\prime},t^{\prime}\right)  ;\\\left(  u,v\right)  \neq\left(
t^{\prime},s^{\prime}\right)  }}\operatorname*{has}\nolimits_{u,v}%
\overrightarrow{a}\cdot\left(  u,v\right)  +\operatorname*{has}%
\nolimits_{s^{\prime},t^{\prime}}\overrightarrow{a}\cdot\left(  s^{\prime
},t^{\prime}\right)  -\left(  s^{\prime},t^{\prime}\right)
+\operatorname*{has}\nolimits_{t^{\prime},s^{\prime}}\overrightarrow{a}%
\cdot\left(  t^{\prime},s^{\prime}\right)  +\left(  t^{\prime},s^{\prime
}\right) \\
&  =\underbrace{\sum_{\substack{\left(  u,v\right)  \in\mathfrak{N};\\\left(
u,v\right)  \neq\left(  s^{\prime},t^{\prime}\right)  ;\\\left(  u,v\right)
\neq\left(  t^{\prime},s^{\prime}\right)  }}\operatorname*{has}\nolimits_{u,v}%
\overrightarrow{a}\cdot\left(  u,v\right)  +\operatorname*{has}%
\nolimits_{s^{\prime},t^{\prime}}\overrightarrow{a}\cdot\left(  s^{\prime
},t^{\prime}\right)  +\operatorname*{has}\nolimits_{t^{\prime},s^{\prime}%
}\overrightarrow{a}\cdot\left(  t^{\prime},s^{\prime}\right)  }%
_{\substack{=\sum_{\left(  u,v\right)  \in\mathfrak{N}}\operatorname*{has}%
\nolimits_{u,v}\overrightarrow{a}\cdot\left(  u,v\right)  \\\text{(since
}\left(  s^{\prime},t^{\prime}\right)  \neq\left(  t^{\prime},s^{\prime
}\right)  \text{)}}}-\left(  s^{\prime},t^{\prime}\right)  +\left(  t^{\prime
},s^{\prime}\right) \\
&  =\underbrace{\sum_{\left(  u,v\right)  \in\mathfrak{N}}\operatorname*{has}%
\nolimits_{u,v}\overrightarrow{a}\cdot\left(  u,v\right)  }%
_{=\operatorname*{Has}\overrightarrow{a}}-\left(  s^{\prime},t^{\prime
}\right)  +\left(  t^{\prime},s^{\prime}\right)  =\operatorname*{Has}%
\overrightarrow{a}-\left(  s^{\prime},t^{\prime}\right)  +\left(  t^{\prime
},s^{\prime}\right)  .
\end{align*}
This proves Theorem \ref{thm.has}.
\end{proof}

\section{The proof of Theorem \ref{thm.BCL}}

We are now ready to establish Theorem \ref{thm.BCL}:

\begin{proof}
[Proof of Theorem \ref{thm.BCL}.]We shall use the \emph{Iverson bracket
notation}: i.e., if $\mathcal{A}$ is any logical statement, then we shall
write $\left[  \mathcal{A}\right]  $ for the integer $%
\begin{cases}
1, & \text{if }\mathcal{A}\text{ is true};\\
0, & \text{if }\mathcal{A}\text{ is false}%
\end{cases}
$.

For every $z\in\mathbb{Z}\left[  \mathfrak{N}\right]  $ and $n\in\mathfrak{N}%
$, we let $\operatorname*{coord}\nolimits_{n}z\in\mathbb{Z}$ be the
$n$-coordinate of $z$ (with respect to the basis $\mathfrak{N}$ of
$\mathbb{Z}\left[  \mathfrak{N}\right]  $).

For every $z\in\mathbb{Z}\left[  \mathfrak{N}\right]  $ and $N\subseteq
\mathfrak{N}$, we set $\operatorname*{coord}\nolimits_{N}z=\sum_{n\in
N}\operatorname*{coord}\nolimits_{n}z$.

We have $c=\left[  \left(  s,t\right)  \right]  $, thus $c_{\mathfrak{N}%
}=\left[  \left[  \left(  s,t\right)  \right]  \right]  $ and
$c^{\operatorname{op}}=\left[  \left(  t,s\right)  \right]  $. From the latter
equality, we obtain $\left(  c^{\operatorname*{op}}\right)  _{\mathfrak{N}%
}=\left[  \left[  \left(  t,s\right)  \right]  \right]  $.

Let $\overrightarrow{c_{1}},\overrightarrow{c_{2}},\ldots
,\overrightarrow{c_{k}},\overrightarrow{c_{k+1}}$ be the vertices on the cycle
$C$ (listed in the order they are encountered when we traverse the cycle,
starting at some arbitrarily chosen vertex on the cycle and going until we
return to the starting point). Thus:

\begin{itemize}
\item We have $\overrightarrow{c_{k+1}}=\overrightarrow{c_{1}}$.

\item There is an arc from $\overrightarrow{c_{i}}$ to
$\overrightarrow{c_{i+1}}$ for every $i\in\left\{  1,2,\ldots,k\right\}  $.
\end{itemize}

Fix $i\in\left\{  1,2,\ldots,k\right\}  $. Then, there is an arc from
$\overrightarrow{c_{i}}$ to $\overrightarrow{c_{i+1}}$. In other words, there
exists some $\left(  s_{i},t_{i}\right)  \in\mathfrak{M}$ such that
$\overrightarrow{c_{i+1}}$ is obtained from $\overrightarrow{c_{i}}$ by an
$\left(  s_{i},t_{i}\right)  $-braid move. Consider this $\left(  s_{i}%
,t_{i}\right)  $. Thus,%
\begin{equation}
\text{the color of the arc from }\overrightarrow{c_{i}}\text{ to
}\overrightarrow{c_{i+1}}\text{ is }\left[  \left(  s_{i},t_{i}\right)
\right]  . \label{pf.thm.BCL.a.color}%
\end{equation}
Proposition \ref{prop.Invsles} \textbf{(b)} (applied to $\overrightarrow{c_{i}%
}$, $\overrightarrow{c_{i+1}}$, $s_{i}$ and $t_{i}$ instead of
$\overrightarrow{a}$, $\overrightarrow{b}$, $s$ and $t$) shows that there
exists a $q\in W$ such that $\operatorname*{Invs}\overrightarrow{c_{i+1}}$ is
obtained from $\operatorname*{Invs}\overrightarrow{c_{i}}$ by replacing a
particular factor of the form $q\rho_{s_{i},t_{i}}q^{-1}$ by its reversal. Let
us denote this $q$ by $q_{i}$. Set $s_{i}^{\prime}=q_{i}s_{i}q_{i}^{-1}$ and
$t_{i}^{\prime}=q_{i}t_{i}q_{i}^{-1}$. Thus, $s_{i}^{\prime}\neq t_{i}%
^{\prime}$ (since $s_{i}\neq t_{i}$) and $m_{s_{i}^{\prime},t_{i}^{\prime}%
}=m_{s_{i},t_{i}}<\infty$ (since $\left(  s_{i},t_{i}\right)  \in\mathfrak{M}%
$). Also, the definitions of $s_{i}^{\prime}$ and $t_{i}^{\prime}$ yield
$\left(  s_{i}^{\prime},t_{i}^{\prime}\right)  =\left(  q_{i}s_{i}q_{i}%
^{-1},q_{i}t_{i}q_{i}^{-1}\right)  =q_{i}\underbrace{\left(  s_{i}%
,t_{i}\right)  }_{\in\mathfrak{M}}q_{i}^{-1}\in q_{i}\mathfrak{M}q_{i}%
^{-1}\subseteq\mathfrak{N}$. From $s_{i}^{\prime}=q_{i}s_{i}q_{i}^{-1}$ and
$t_{i}^{\prime}=q_{i}t_{i}q_{i}^{-1}$, we obtain $\left(  s_{i}^{\prime}%
,t_{i}^{\prime}\right)  \approx\left(  s_{i},t_{i}\right)  $.

We shall now show that%
\begin{equation}
\operatorname*{coord}\nolimits_{c_{\mathfrak{N}}}\left(  \operatorname*{Has}%
\overrightarrow{c_{i+1}}-\operatorname*{Has}\overrightarrow{c_{i}}\right)
=\left[  \left[  \left(  s_{i},t_{i}\right)  \right]  =c^{\operatorname*{op}%
}\right]  -\left[  \left[  \left(  s_{i},t_{i}\right)  \right]  =c\right]  .
\label{pf.thm.BCL.a.hasdiff1}%
\end{equation}

\textit{Proof of (\ref{pf.thm.BCL.a.hasdiff1}):} We have the following chain
of logical equivalences:%
\begin{align*}
&  \ \left(  \left(  t_{i}^{\prime},s_{i}^{\prime}\right)  \in
\underbrace{c_{\mathfrak{N}}}_{=\left[  \left[  \left(  s,t\right)  \right]
\right]  }\right) \\
&  \Longleftrightarrow\ \left(  \left(  t_{i}^{\prime},s_{i}^{\prime}\right)
\in\left[  \left[  \left(  s,t\right)  \right]  \right]  \right)
\ \Longleftrightarrow\ \left(  \left(  t_{i}^{\prime},s_{i}^{\prime}\right)
\approx\left(  s,t\right)  \right)  \ \Longleftrightarrow\ \left(  \left(
s_{i}^{\prime},t_{i}^{\prime}\right)  \approx\left(  t,s\right)  \right) \\
&  \Longleftrightarrow\ \left(  \left(  s_{i},t_{i}\right)  \approx\left(
t,s\right)  \right)  \ \ \ \ \ \ \ \ \ \ \left(  \text{since }\left(
s_{i}^{\prime},t_{i}^{\prime}\right)  \approx\left(  s_{i},t_{i}\right)
\right) \\
&  \Longleftrightarrow\ \left(  \left(  s_{i},t_{i}\right)  \sim\left(
t,s\right)  \right)  \ \ \ \ \ \ \ \ \ \ \left(  \text{since the restriction
of the relation }\approx\text{ to }\mathfrak{M}\text{ is }\sim\right) \\
&  \Longleftrightarrow\ \left(  \left(  s_{i},t_{i}\right)  \in
\underbrace{\left[  \left(  t,s\right)  \right]  }_{=c^{\operatorname*{op}}%
}\right)  \ \Longleftrightarrow\ \left(  \left(  s_{i},t_{i}\right)  \in
c^{\operatorname*{op}}\right)  \ \Longleftrightarrow\ \left(  \left[  \left(
s_{i},t_{i}\right)  \right]  =c^{\operatorname*{op}}\right)  .
\end{align*}
Hence,%
\begin{equation}
\left[  \left(  t_{i}^{\prime},s_{i}^{\prime}\right)  \in c_{\mathfrak{N}%
}\right]  =\left[  \left[  \left(  s_{i},t_{i}\right)  \right]
=c^{\operatorname*{op}}\right]  . \label{pf.thm.BCL.a.hasdiff1.eq1}%
\end{equation}

Also, we have the following chain of logical equivalences:%
\begin{align*}
&  \ \left(  \left(  s_{i}^{\prime},t_{i}^{\prime}\right)  \in
\underbrace{c_{\mathfrak{N}}}_{=\left[  \left[  \left(  s,t\right)  \right]
\right]  }\right) \\
&  \Longleftrightarrow\ \left(  \left(  s_{i}^{\prime},t_{i}^{\prime}\right)
\in\left[  \left[  \left(  s,t\right)  \right]  \right]  \right)
\ \Longleftrightarrow\ \left(  \left(  s_{i}^{\prime},t_{i}^{\prime}\right)
\approx\left(  s,t\right)  \right) \\
&  \Longleftrightarrow\ \left(  \left(  s_{i},t_{i}\right)  \approx\left(
s,t\right)  \right)  \ \ \ \ \ \ \ \ \ \ \left(  \text{since }\left(
s_{i}^{\prime},t_{i}^{\prime}\right)  \approx\left(  s_{i},t_{i}\right)
\right) \\
&  \Longleftrightarrow\ \left(  \left(  s_{i},t_{i}\right)  \sim\left(
s,t\right)  \right)  \ \ \ \ \ \ \ \ \ \ \left(  \text{since the restriction
of the relation }\approx\text{ to }\mathfrak{M}\text{ is }\sim\right) \\
&  \Longleftrightarrow\ \left(  \left(  s_{i},t_{i}\right)  \in
\underbrace{\left[  \left(  s,t\right)  \right]  }_{=c}\right)
\ \Longleftrightarrow\ \left(  \left(  s_{i},t_{i}\right)  \in c\right)
\ \Longleftrightarrow\ \left(  \left[  \left(  s_{i},t_{i}\right)  \right]
=c\right)  .
\end{align*}
Hence,%
\begin{equation}
\left[  \left(  s_{i}^{\prime},t_{i}^{\prime}\right)  \in c_{\mathfrak{N}%
}\right]  =\left[  \left[  \left(  s_{i},t_{i}\right)  \right]  =c\right]  .
\label{pf.thm.BCL.a.hasdiff1.eq2}%
\end{equation}

Applying (\ref{eq.thm.has.a}) to $\overrightarrow{c_{i}}$,
$\overrightarrow{c_{i+1}}$, $s_{i}$, $t_{i}$, $q_{i}$, $s_{i}^{\prime}$ and
$t_{i}^{\prime}$ instead of $\overrightarrow{a}$, $\overrightarrow{b}$, $s$,
$t$, $q$, $s^{\prime}$ and $t^{\prime}$, we obtain $\operatorname*{Has}%
\overrightarrow{c_{i+1}}=\operatorname*{Has}\overrightarrow{c_{i}}-\left(
s_{i}^{\prime},t_{i}^{\prime}\right)  +\left(  t_{i}^{\prime},s_{i}^{\prime
}\right)  $. In other words, $\operatorname*{Has}\overrightarrow{c_{i+1}%
}-\operatorname*{Has}\overrightarrow{c_{i}}=\left(  t_{i}^{\prime}%
,s_{i}^{\prime}\right)  -\left(  s_{i}^{\prime},t_{i}^{\prime}\right)  $.
Thus,%
\begin{align*}
&  \operatorname*{coord}\nolimits_{c_{\mathfrak{N}}}\left(
\operatorname*{Has}\overrightarrow{c_{i+1}}-\operatorname*{Has}%
\overrightarrow{c_{i}}\right) \\
&  =\operatorname*{coord}\nolimits_{c_{\mathfrak{N}}}\left(  \left(
t_{i}^{\prime},s_{i}^{\prime}\right)  -\left(  s_{i}^{\prime},t_{i}^{\prime
}\right)  \right)  =\underbrace{\operatorname*{coord}%
\nolimits_{c_{\mathfrak{N}}}\left(  t_{i}^{\prime},s_{i}^{\prime}\right)
}_{\substack{=\left[  \left(  t_{i}^{\prime},s_{i}^{\prime}\right)  \in
c_{\mathfrak{N}}\right]  \\=\left[  \left[  \left(  s_{i},t_{i}\right)
\right]  =c^{\operatorname*{op}}\right]  \\\text{(by
(\ref{pf.thm.BCL.a.hasdiff1.eq1}))}}}-\underbrace{\operatorname*{coord}%
\nolimits_{c_{\mathfrak{N}}}\left(  s_{i}^{\prime},t_{i}^{\prime}\right)
}_{\substack{=\left[  \left(  s_{i}^{\prime},t_{i}^{\prime}\right)  \in
c_{\mathfrak{N}}\right]  \\=\left[  \left[  \left(  s_{i},t_{i}\right)
\right]  =c\right]  \\\text{(by (\ref{pf.thm.BCL.a.hasdiff1.eq2}))}}}\\
&  =\left[  \left[  \left(  s_{i},t_{i}\right)  \right]
=c^{\operatorname*{op}}\right]  -\left[  \left[  \left(  s_{i},t_{i}\right)
\right]  =c\right]  .
\end{align*}
This proves (\ref{pf.thm.BCL.a.hasdiff1}).

Now, let us forget that we fixed $i$. Thus, for every $i\in\left\{
1,2,\ldots,k\right\}  $, we have defined $\left(  s_{i},t_{i}\right)
\in\mathfrak{M}$ satisfying (\ref{pf.thm.BCL.a.color}) and
(\ref{pf.thm.BCL.a.hasdiff1}).

We have $\operatorname*{coord}\nolimits_{c_{\mathfrak{N}}}\left(
\operatorname*{Has}\overrightarrow{c_{i+1}}-\operatorname*{Has}%
\overrightarrow{c_{i}}\right)  =\operatorname*{coord}%
\nolimits_{c_{\mathfrak{N}}}\left(  \operatorname*{Has}\overrightarrow{c_{i+1}%
}\right)  -\operatorname*{coord}\nolimits_{c_{\mathfrak{N}}}\left(
\operatorname*{Has}\overrightarrow{c_{i}}\right)  $ for all $i\in\left\{
1,2,\ldots,k\right\}  $. Hence,%
\begin{align*}
&  \sum_{i=1}^{k}\operatorname*{coord}\nolimits_{c_{\mathfrak{N}}}\left(
\operatorname*{Has}\overrightarrow{c_{i+1}}-\operatorname*{Has}%
\overrightarrow{c_{i}}\right) \\
&  =\sum_{i=1}^{k}\left(  \operatorname*{coord}\nolimits_{c_{\mathfrak{N}}%
}\left(  \operatorname*{Has}\overrightarrow{c_{i+1}}\right)
-\operatorname*{coord}\nolimits_{c_{\mathfrak{N}}}\left(  \operatorname*{Has}%
\overrightarrow{c_{i}}\right)  \right)  =0
\end{align*}
(by the telescope principle). Hence,%
\begin{align*}
0  &  =\sum_{i=1}^{k}\operatorname*{coord}\nolimits_{c_{\mathfrak{N}}}\left(
\operatorname*{Has}\overrightarrow{c_{i+1}}-\operatorname*{Has}%
\overrightarrow{c_{i}}\right) \\
&  =\sum_{i=1}^{k}\left(  \left[  \left[  \left(  s_{i},t_{i}\right)  \right]
=c^{\operatorname*{op}}\right]  -\left[  \left[  \left(  s_{i},t_{i}\right)
\right]  =c\right]  \right)  \ \ \ \ \ \ \ \ \ \ \left(  \text{by
(\ref{pf.thm.BCL.a.hasdiff1})}\right) \\
&  =\sum_{i=1}^{k}\left[  \left[  \left(  s_{i},t_{i}\right)  \right]
=c^{\operatorname*{op}}\right]  -\sum_{i=1}^{k}\left[  \left[  \left(
s_{i},t_{i}\right)  \right]  =c\right]  .
\end{align*}
Comparing this with%
\begin{align*}
&  \left(  \text{the number of arcs colored }c^{\operatorname*{op}}\text{
appearing in }C\right) \\
&  \ \ \ \ \ \ \ \ \ \ -\left(  \text{the number of arcs colored }c\text{
appearing in }C\right) \\
&  =\sum_{i=1}^{k}\left[  \left(  \text{the color of the arc from
}\overrightarrow{c_{i}}\text{ to }\overrightarrow{c_{i+1}}\right)
=c^{\operatorname*{op}}\right] \\
&  \ \ \ \ \ \ \ \ \ \ -\sum_{i=1}^{k}\left[  \left(  \text{the color of the
arc from }\overrightarrow{c_{i}}\text{ to }\overrightarrow{c_{i+1}}\right)
=c\right] \\
&  =\sum_{i=1}^{k}\left[  \left[  \left(  s_{i},t_{i}\right)  \right]
=c^{\operatorname*{op}}\right]  -\sum_{i=1}^{k}\left[  \left[  \left(
s_{i},t_{i}\right)  \right]  =c\right]  \ \ \ \ \ \ \ \ \ \ \left(  \text{by
(\ref{pf.thm.BCL.a.color})}\right)  ,
\end{align*}
we obtain%
\begin{align*}
&  \left(  \text{the number of arcs colored }c^{\operatorname*{op}}\text{
appearing in }C\right) \\
&  \ \ \ \ \ \ \ \ \ \ -\left(  \text{the number of arcs colored }c\text{
appearing in }C\right) \\
&  =0.
\end{align*}
In other words, the number of arcs colored $c$ appearing in $C$ equals the
number of arcs colored $c^{\operatorname*{op}}$ appearing in $C$. This proves
Theorem \ref{thm.BCL} \textbf{(a)}. \medskip

\textbf{(b)} If $c\neq c^{\operatorname*{op}}$, then Theorem \ref{thm.BCL}
\textbf{(b)} follows immediately from Theorem \ref{thm.BCL} \textbf{(a)}.
Thus, for the rest of this proof, assume that $c=c^{\operatorname*{op}}$
(without loss of generality).

We have $\left[  \left(  s,t\right)  \right]  =c=c^{\operatorname*{op}%
}=\left[  \left(  t,s\right)  \right]  $, so that $\left(  t,s\right)
\sim\left(  s,t\right)  $. Hence, $\left(  t,s\right)  \approx\left(
s,t\right)  $ (since $\sim$ is the restriction of the relation $\approx$ to
$\mathfrak{M}$).

Fix some total order on the set $T$. Let $d$ be the subset $\left\{  \left(
u,v\right)  \in c_{\mathfrak{N}}\ \mid\ u<v\right\}  $ of $c_{\mathfrak{N}}$.

Fix $i\in\left\{  1,2,\ldots,k\right\}  $. We shall now show that%
\begin{equation}
\operatorname*{coord}\nolimits_{d}\left(  \operatorname*{Has}%
\overrightarrow{c_{i+1}}-\operatorname*{Has}\overrightarrow{c_{i}}\right)
\equiv\left[  \left[  \left(  s_{i},t_{i}\right)  \right]  =c\right]
\operatorname{mod}2. \label{pf.thm.BCL.b.hasdiff1}%
\end{equation}

\textit{Proof of (\ref{pf.thm.BCL.b.hasdiff1}):} Define $q_{i}$,
$s_{i}^{\prime}$ and $t_{i}^{\prime}$ as before. We have $s_{i}^{\prime}\neq
t_{i}^{\prime}$. Hence, either $s_{i}^{\prime}<t_{i}^{\prime}$ or
$t_{i}^{\prime}<s_{i}^{\prime}$.

We have the following equivalences:%
\begin{align}
\left(  \left(  t_{i}^{\prime},s_{i}^{\prime}\right)  \in c_{\mathfrak{N}%
}\right)  \  &  \Longleftrightarrow\ \left(  \left(  t_{i}^{\prime}%
,s_{i}^{\prime}\right)  \in\left[  \left[  \left(  s,t\right)  \right]
\right]  \right)  \ \ \ \ \ \ \ \ \ \ \left(  \text{since }c_{\mathfrak{N}%
}=\left[  \left[  \left(  s,t\right)  \right]  \right]  \right) \nonumber\\
&  \Longleftrightarrow\ \left(  \left(  t_{i}^{\prime},s_{i}^{\prime}\right)
\approx\left(  s,t\right)  \right)  \ \Longleftrightarrow\ \left(
s_{i}^{\prime},t_{i}^{\prime}\right)  \approx\left(  t,s\right)
\ \Longleftrightarrow\ \left(  \left(  s_{i},t_{i}\right)  \approx\left(
s,t\right)  \right) \nonumber\\
&  \ \ \ \ \ \ \ \ \ \ \ \left(  \text{since }\left(  s_{i}^{\prime}%
,t_{i}^{\prime}\right)  \approx\left(  s_{i},t_{i}\right)  \text{ and }\left(
t,s\right)  \approx\left(  s,t\right)  \right) \nonumber\\
&  \Longleftrightarrow\ \left(  \left(  s_{i},t_{i}\right)  \sim\left(
s,t\right)  \right)  \label{pf.thm.BCL.b.hasdiff1.pf.equiv1}%
\end{align}
(since the restriction of the relation $\approx$ to $\mathfrak{M}$ is $\sim$)
and%
\begin{align}
\left(  \left(  s_{i}^{\prime},t_{i}^{\prime}\right)  \in c_{\mathfrak{N}%
}\right)  \  &  \Longleftrightarrow\ \left(  \left(  s_{i}^{\prime}%
,t_{i}^{\prime}\right)  \in\left[  \left[  \left(  s,t\right)  \right]
\right]  \right)  \ \ \ \ \ \ \ \ \ \ \left(  \text{since }c_{\mathfrak{N}%
}=\left[  \left[  \left(  s,t\right)  \right]  \right]  \right) \nonumber\\
&  \Longleftrightarrow\ \left(  \left(  s_{i}^{\prime},t_{i}^{\prime}\right)
\approx\left(  s,t\right)  \right)  \ \Longleftrightarrow\ \left(  \left(
s_{i},t_{i}\right)  \approx\left(  s,t\right)  \right) \nonumber\\
&  \Longleftrightarrow\ \left(  \left(  s_{i},t_{i}\right)  \sim\left(
s,t\right)  \right)  . \label{pf.thm.BCL.b.hasdiff1.pf.equiv2}%
\end{align}

Applying (\ref{eq.thm.has.a}) to $\overrightarrow{c_{i}}$,
$\overrightarrow{c_{i+1}}$, $s_{i}$, $t_{i}$, $q_{i}$, $s_{i}^{\prime}$ and
$t_{i}^{\prime}$ instead of $\overrightarrow{a}$, $\overrightarrow{b}$, $s$,
$t$, $q$, $s^{\prime}$ and $t^{\prime}$, we obtain $\operatorname*{Has}%
\overrightarrow{c_{i+1}}=\operatorname*{Has}\overrightarrow{c_{i}}-\left(
s_{i}^{\prime},t_{i}^{\prime}\right)  +\left(  t_{i}^{\prime},s_{i}^{\prime
}\right)  $. In other words, $\operatorname*{Has}\overrightarrow{c_{i+1}%
}-\operatorname*{Has}\overrightarrow{c_{i}}=\left(  t_{i}^{\prime}%
,s_{i}^{\prime}\right)  -\left(  s_{i}^{\prime},t_{i}^{\prime}\right)  $.
Thus,%
\begin{align*}
&  \operatorname*{coord}\nolimits_{d}\left(  \operatorname*{Has}%
\overrightarrow{c_{i+1}}-\operatorname*{Has}\overrightarrow{c_{i}}\right) \\
&  =\operatorname*{coord}\nolimits_{d}\left(  \left(  t_{i}^{\prime}%
,s_{i}^{\prime}\right)  -\left(  s_{i}^{\prime},t_{i}^{\prime}\right)
\right)  =\operatorname*{coord}\nolimits_{d}\left(  t_{i}^{\prime}%
,s_{i}^{\prime}\right)  -\operatorname*{coord}\nolimits_{d}\left(
s_{i}^{\prime},t_{i}^{\prime}\right) \\
&  =\left[  \left(  t_{i}^{\prime},s_{i}^{\prime}\right)  \in d\right]
-\left[  \left(  s_{i}^{\prime},t_{i}^{\prime}\right)  \in d\right] \\
&  \equiv\left[  \left(  t_{i}^{\prime},s_{i}^{\prime}\right)  \in d\right]
+\left[  \left(  s_{i}^{\prime},t_{i}^{\prime}\right)  \in d\right] \\
&  =\left[  \left(  t_{i}^{\prime},s_{i}^{\prime}\right)  \in c_{\mathfrak{N}%
}\text{ and }t_{i}^{\prime}<s_{i}^{\prime}\right]  +\left[  \left(
s_{i}^{\prime},t_{i}^{\prime}\right)  \in c_{\mathfrak{N}}\text{ and }%
s_{i}^{\prime}<t_{i}^{\prime}\right] \\
&  \ \ \ \ \ \ \ \ \ \ \left(  \text{since a pair }\left(  u,v\right)  \text{
belongs to }d\text{ if and only if }\left(  u,v\right)  \in c_{\mathfrak{N}%
}\text{ and }u<v\right) \\
&  =\left[  \left(  s_{i},t_{i}\right)  \sim\left(  s,t\right)  \text{ and
}t_{i}^{\prime}<s_{i}^{\prime}\right]  +\left[  \left(  s_{i},t_{i}\right)
\sim\left(  s,t\right)  \text{ and }s_{i}^{\prime}<t_{i}^{\prime}\right] \\
&  \ \ \ \ \ \ \ \ \ \ \left(  \text{by the equivalences
(\ref{pf.thm.BCL.b.hasdiff1.pf.equiv1}) and
(\ref{pf.thm.BCL.b.hasdiff1.pf.equiv2})}\right) \\
&  =\left[  \left(  s_{i},t_{i}\right)  \sim\left(  s,t\right)  \right]
\ \ \ \ \ \ \ \ \ \ \left(  \text{because either }s_{i}^{\prime}<t_{i}%
^{\prime}\text{ or }t_{i}^{\prime}<s_{i}^{\prime}\right) \\
&  =\left[  \left[  \left(  s_{i},t_{i}\right)  \right]  =\left[  \left(
s,t\right)  \right]  \right]  =\left[  \left[  \left(  s_{i},t_{i}\right)
\right]  =c\right]  \operatorname{mod}2\ \ \ \ \ \ \ \ \ \ \left(  \text{since
}\left[  \left(  s,t\right)  \right]  =c\right)  .
\end{align*}
This proves (\ref{pf.thm.BCL.b.hasdiff1}).

Now, $\operatorname*{coord}\nolimits_{d}\left(  \operatorname*{Has}%
\overrightarrow{c_{i+1}}-\operatorname*{Has}\overrightarrow{c_{i}}\right)
=\operatorname*{coord}\nolimits_{d}\left(  \operatorname*{Has}%
\overrightarrow{c_{i+1}}\right)  -\operatorname*{coord}\nolimits_{d}\left(
\operatorname*{Has}\overrightarrow{c_{i}}\right)  $ for each $i\in\left\{
1,2,\ldots,k\right\}  $; hence,%
\[
\sum_{i=1}^{k}\operatorname*{coord}\nolimits_{d}\left(  \operatorname*{Has}%
\overrightarrow{c_{i+1}}-\operatorname*{Has}\overrightarrow{c_{i}}\right)
=\sum_{i=1}^{k}\left(  \operatorname*{coord}\nolimits_{d}\left(
\operatorname*{Has}\overrightarrow{c_{i+1}}\right)  -\operatorname*{coord}%
\nolimits_{d}\left(  \operatorname*{Has}\overrightarrow{c_{i}}\right)
\right)  =0
\]
(by the telescope principle). Hence,%
\begin{align*}
0  &  =\sum_{i=1}^{k}\operatorname*{coord}\nolimits_{d}\left(
\operatorname*{Has}\overrightarrow{c_{i+1}}-\operatorname*{Has}%
\overrightarrow{c_{i}}\right) \\
&  \equiv\sum_{i=1}^{k}\left[  \left[  \left(  s_{i},t_{i}\right)  \right]
=c\right]  \ \ \ \ \ \ \ \ \ \ \left(  \text{by (\ref{pf.thm.BCL.b.hasdiff1}%
)}\right) \\
&  =\sum_{i=1}^{k}\left[  \left(  \text{the color of the arc from
}\overrightarrow{c_{i}}\text{ to }\overrightarrow{c_{i+1}}\right)  =c\right]
\ \ \ \ \ \ \ \ \ \ \left(  \text{by (\ref{pf.thm.BCL.a.color})}\right) \\
&  =\left(  \text{the number of arcs colored }c\text{ appearing in }C\right)
\operatorname{mod}2.
\end{align*}
Thus, the number of arcs colored $c$ appearing in $C$ is even. In other words,
the number of arcs whose color belongs to $\left\{  c\right\}  $ appearing in
$C$ is even. In other words, the number of arcs whose color belongs to
$\left\{  c,c^{\operatorname*{op}}\right\}  $ appearing in $C$ is even (since
$\left\{  c,\underbrace{c^{\operatorname*{op}}}_{=c}\right\}  =\left\{
c,c\right\}  =\left\{  c\right\}  $). This proves Theorem \ref{thm.BCL}
\textbf{(b)}.
\end{proof}

\section{Open questions}

Theorem \ref{thm.BCL} is a statement about reduced expressions. As with all
such statements, one can wonder whether a generalization to \textquotedblleft
non-reduced\textquotedblright\ expressions would still be true. If $w$ is an
element of $W$, then an \emph{expression} for $w$ means a $k$-tuple $\left(
s_{1},s_{2},\ldots,s_{k}\right)  $ of elements of $S$ such that $w=s_{1}%
s_{2}\cdots s_{k}$. Definition \ref{def.braid} can be applied verbatim to
arbitrary expressions, leading to the concept of an $\left(  s,t\right)
$-braid move. Finally, for every $w\in W$, we define a directed graph
$\mathcal{E}\left(  w\right)  $ in the same way as we defined $\mathcal{R}%
\left(  w\right)  $ in Definition \ref{def.R}, but with the word
\textquotedblleft reduced\textquotedblright\ removed everywhere. This directed
graph $\mathcal{E}\left(  w\right)  $ will be infinite (in general) and
consist of many connected components (one of which is $\mathcal{R}\left(
w\right)  $), but we can still inquire about its cycles. We conjecture the
following generalization of Theorem \ref{thm.BCL}:

\begin{conjecture}
\label{conj.E(w)}Let $w\in W$. Theorem \ref{thm.BCL} is still valid if we
replace $\mathcal{R}\left(  w\right)  $ by $\mathcal{E}\left(  w\right)  $.
\end{conjecture}

A further, slightly lateral, generalization concerns a kind of
\textquotedblleft spin extension\textquotedblright\ of a Coxeter group:

\begin{conjecture}
\label{conj.spin}For every $\left(  s,t\right)  \in\mathfrak{M}$, let
$c_{s,t}$ be an element of $\left\{  1,-1\right\}  $. Assume that
$c_{s,t}=c_{s^{\prime},t^{\prime}}$ for any two elements $\left(  s,t\right)
$ and $\left(  s^{\prime},t^{\prime}\right)  $ of $\mathfrak{M}$ satisfying
$\left(  s,t\right)  \sim\left(  s^{\prime},t^{\prime}\right)  $. Assume
furthermore that $c_{s,t}=c_{t,s}$ for each $\left(  s,t\right)
\in\mathfrak{M}$. Let $W^{\prime}$ be the group with the following generators
and relations:

\textit{Generators:} the elements $s\in S$ and an extra generator $q$.

\textit{Relations:}%
\begin{align*}
s^{2}  &  =1\ \ \ \ \ \ \ \ \ \ \text{for every }s\in S;\\
q^{2}  &  =1;\\
qs  &  =sq\ \ \ \ \ \ \ \ \ \ \text{for every }s\in S;\\
\left(  st\right)  ^{m_{s,t}}  &  =1\ \ \ \ \ \ \ \ \ \ \text{for every
}\left(  s,t\right)  \in\mathfrak{M}\text{ satisfying }c_{s,t}=1;\\
\left(  st\right)  ^{m_{s,t}}  &  =q\ \ \ \ \ \ \ \ \ \ \text{for every
}\left(  s,t\right)  \in\mathfrak{M}\text{ satisfying }c_{s,t}=-1.
\end{align*}

There is clearly a surjective group homomorphism $\pi:W^{\prime}\rightarrow W$
sending each $s\in S$ to $s$, and sending $q$ to $1$. There is also a group
homomorphism $\iota:\mathbb{Z}/2\mathbb{Z}\rightarrow W^{\prime}$ which sends
the generator of $\mathbb{Z}/2\mathbb{Z}$ to $q$. Then, the sequence%
\begin{equation}
1\longrightarrow\mathbb{Z}/2\overset{\iota}{\longrightarrow}W^{\prime
}\overset{\pi}{\longrightarrow}W\longrightarrow1\label{eq.conj.spin.seq}%
\end{equation}
is exact. Equivalently, $\iota$ is injective. Equivalently, $\left\vert
\operatorname*{Ker}\pi\right\vert =2$.
\end{conjecture}

(Note that exactness of the sequence (\ref{eq.conj.spin.seq}) at $W^{\prime}$
and at $W$ is easy.)

If Conjecture \ref{conj.spin} holds, then so does Conjecture \ref{conj.E(w)}
\textbf{(b)} (that is, Theorem \ref{thm.BCL} \textbf{(b)} holds with
$\mathcal{R}\left(  w\right)  $ replaced by $\mathcal{E}\left(  w\right)  $).
Indeed, assume Conjecture \ref{conj.spin} to hold. Let $c\in\mathfrak{M}/\sim$
be an equivalence class. For any $\left(  u,v\right)  \in\mathfrak{M}$, define%
\[
c_{u,v}=%
\begin{cases}
-1, & \text{if }\left(  u,v\right)  \in c\text{ or }\left(  v,u\right)  \in
c;\\
1, & \text{otherwise}%
\end{cases}
.
\]
Thus, a group $W^{\prime}$ is defined. Pick any section $\mathbf{s}%
:W\rightarrow W^{\prime}$ (in the category of sets) of the projection
$\pi:W^{\prime}\rightarrow W$. If $w\in W$, and if $\left(  s_{1},s_{2}%
,\ldots,s_{k}\right)  $ is an expression of $w$, then the product $s_{1}%
s_{2}\cdots s_{k}$ formed in $W^{\prime}$ will either be $\mathbf{s}\left(
w\right)  $ or $q\mathbf{s}\left(  w\right)  $; and these latter two values
are distinct (by Conjecture \ref{conj.spin}). We can then define the
\emph{sign} of the expression $\left(  s_{1},s_{2},\ldots,s_{k}\right)  $ to
be $%
\begin{cases}
1, & \text{if }s_{1}s_{2}\cdots s_{k}=\mathbf{s}\left(  w\right)  ;\\
-1, & \text{if }s_{1}s_{2}\cdots s_{k}=q\mathbf{s}\left(  w\right)
\end{cases}
\in\left\{  1,-1\right\}  $. The sign of an expression switches when we apply
a braid move whose arc's color belongs to $\left\{  c,c^{\operatorname*{op}%
}\right\}  $, but stays unchanged when we apply a braid move of any other
color. Conjecture \ref{conj.E(w)} \textbf{(b)} then follows by a simple parity argument.

The construction of $W^{\prime}$ in Conjecture \ref{conj.spin} generalizes the
construction of one of the two \emph{spin symmetric groups} (up to a
substitution). We suspect that Conjecture \ref{conj.spin} could be proven by
constructing a \textquotedblleft regular representation\textquotedblright, and
this would then yield an alternative proof of Theorem \ref{thm.BCL}
\textbf{(b)}.

\end{document}